\documentclass[11pt]{article}
\usepackage{amsmath,amssymb,amsthm,amscd,enumerate}
\usepackage{xcolor}
\usepackage[all]{xy}
\usepackage{graphics}

\definecolor{grn}{rgb}{0,0.3,0}
\definecolor{mrn}{rgb}{1,0,0}
\definecolor{blue}{rgb}{0,0,1}

\newtheorem{theorem}{Theorem}[section]
\newtheorem{lemma}[theorem]{Lemma}
\newtheorem{proposition}[theorem]{Proposition}
\newtheorem{corollary}[theorem]{Corollary}

\theoremstyle{plain}

\theoremstyle{definition}

\newtheorem{remark}[theorem]{Remark}

\numberwithin{equation}{section}

\renewcommand{\theenumi}{(\roman{enumi})}
\renewcommand{\labelenumi}{\textup{(\theenumi)}}

\title{On universal property of reciprocal Kirchberg algebras
and uniquely ergodic automorphisms
}
\author{Kengo Matsumoto \\
Department of Mathematics \\
Joetsu University of Education \\
Joetsu 943-8512, Japan
\and
Taro Sogabe \\
Faculty of Advanced Science and Technology \\
Kumamoto University \\
Kurokami 2-40-1, Chuo-ku, Kumamoto \\
860-8555, Japan}

\begin{document}

\maketitle

\date{}

\def\det{{{\operatorname{det}}}}

\begin{abstract}
Reciprocality in Kirchberg algebras with finitely generated K-groups
is regarded as a K-theoretic duality through  K-groups and strong extension groups.
We will prove that   the reciprocal Kirchberg algebra  
has a universal property with respect to some generating $C^*$-subalgebra and a family of generating partial isometries.
By using the universal property, 
we will prove that there exists an aperiodic ergodic automorphism on an arbitrary unital Kirchberg algebra 
with finitely generated K-groups, which has a unique invariant state. The state is pure.
\end{abstract}

{\it Mathematics Subject Classification}:
Primary 46L80; Secondary 19K33.

{\it Keywords and phrases}: K-group, Ext-group,   
$C^*$-algebra,  Kirchberg algebra,  ergodic automorphism,
invariant state.

\newcommand{\Ker}{\operatorname{Ker}}
\newcommand{\sgn}{\operatorname{sgn}}
\newcommand{\Ad}{\operatorname{Ad}}
\newcommand{\ad}{\operatorname{ad}}
\newcommand{\orb}{\operatorname{orb}}
\newcommand{\rank}{\operatorname{rank}}

\def\Re{{\operatorname{Re}}}
\def\det{{{\operatorname{det}}}}
\newcommand{\K}{\operatorname{K}}
\newcommand{\bbK}{\mathbb{K}}
\newcommand{\N}{\mathbb{N}}
\newcommand{\bbC}{\mathbb{C}}
\newcommand{\R}{\mathbb{R}}
\newcommand{\Rp}{{\mathbb{R}}^*_+}
\newcommand{\T}{\mathcal{T}}

\newcommand{\sqK}{\operatorname{K}\!\operatorname{K}}

\newcommand{\Z}{\mathbb{Z}}
\newcommand{\Zp}{{\mathbb{Z}}_+}
\def\AF{{{\operatorname{AF}}}}

\def\TorZ{{{\operatorname{Tor}}^\Z_1}}
\def\Ext{{{\operatorname{Ext}}}}
\def\Exts{\operatorname{Ext}_{\operatorname{s}}}
\def\Extw{\operatorname{Ext}_{\operatorname{w}}}
\def\Extt{\operatorname{Ext}_{\operatorname{t}}}
\def\Ext{{{\operatorname{Ext}}}}
\def\Free{{{\operatorname{Free}}}}

\def\Ks{\operatorname{K}^{\operatorname{s}}}
\def\Kw{\operatorname{K}^{\operatorname{w}}}

\def\OA{{{\mathcal{O}}_A}}
\def\ON{{{\mathcal{O}}_N}}
\def\OAT{{{\mathcal{O}}_{A^t}}}
\def\OAI{{{\mathcal{O}}_{A^\infty}}}
\def\OAIn{{{\mathcal{O}}_{A^\infty_n}}}
\def\OAInone{{{\mathcal{O}}_{A^\infty_{n+1}}}}

\def\OAIN{{{\mathcal{O}}_{A^\infty_N}}}
\def\OATI{{{\mathcal{O}}_{A^{t \infty}}}}

\def\OSA{{{\mathcal{O}}_{S_A}}}

\def\TA{{{\mathcal{T}}_A}}
\def\TAn{{{\mathcal{T}}_{A_n}}}
\def\TAnone{{{\mathcal{T}}_{A_{n+1}}}}
\def\TAN{{{\mathcal{T}}_{A_N}}}

\def\TN{{{\mathcal{T}}_N}}

\def\TAT{{{\mathcal{T}}_{A^t}}}

\def\TB{{{\mathcal{T}}_B}}
\def\TBT{{{\mathcal{T}}_{B^t}}}

\def\A{{\mathcal{A}}}
\def\B{{\mathcal{B}}}
\def\C{{\mathcal{C}}}
\def\O{{\mathcal{O}}}
\def\E{{\mathcal{E}}}
\def\F{{\mathcal{F}}}
\def\L{{\mathcal{L}}}

\def\OaA{{{\mathcal{O}}^a_A}}
\def\OB{{{\mathcal{O}}_B}}
\def\OTA{{{\mathcal{O}}_{\tilde{A}}}}
\def\F{{\mathcal{F}}}
\def\G{{\mathcal{G}}}
\def\FA{{{\mathcal{F}}_A}}
\def\PA{{{\mathcal{P}}_A}}
\def\PI{{{\mathcal{P}}_\infty}}
\def\OI{{{\mathcal{O}}_\infty}}
\def\TI{{{\mathcal{T}}_\infty}}

\def\OHF{{{\mathcal{O}}_{H_\F}}}

\def\whatA{{\widehat{\A}}}
\def\whatOA{\widehat{\mathcal{O}}_A}

\def\calI{\mathcal{I}}

\def\bbC{{\mathbb{C}}}

 \def\U{{\mathcal{U}}}
\def\OF{{{\mathcal{O}}_F}}
\def\DF{{{\mathcal{D}}_F}}
\def\FB{{{\mathcal{F}}_B}}
\def\DA{{{\mathcal{D}}_A}}
\def\DB{{{\mathcal{D}}_B}}
\def\DZ{{{\mathcal{D}}_Z}}

\def\End{{{\operatorname{End}}}}

\def\Ext{{{\operatorname{Ext}}}}
\def\Hom{{{\operatorname{Hom}}}}

\def\Tor{{{\operatorname{Tor}}}}

\def\Max{{{\operatorname{Max}}}}
\def\Max{{{\operatorname{Max}}}}
\def\max{{{\operatorname{max}}}}
\def\KMS{{{\operatorname{KMS}}}}
\def\Per{{{\operatorname{Per}}}}
\def\Out{{{\operatorname{Out}}}}
\def\Aut{{{\operatorname{Aut}}}}
\def\Ad{{{\operatorname{Ad}}}}
\def\Inn{{{\operatorname{Inn}}}}
\def\Int{{{\operatorname{Int}}}}
\def\det{{{\operatorname{det}}}}
\def\exp{{{\operatorname{exp}}}}
\def\nep{{{\operatorname{nep}}}}
\def\sgn{{{\operatorname{sign}}}}
\def\cobdy{{{\operatorname{cobdy}}}}
\def\Ker{{{\operatorname{Ker}}}}
\def\Coker{{{\operatorname{Coker}}}}
\def\Im{{\operatorname{Im}}}

\def\fH{{\frak{H}}}
\def\fK{{\frak{K}}}
\def\tfH{{\widetilde{\frak{H}}}}

\def\ind{{{\operatorname{ind}}}}
\def\Ind{{{\operatorname{Ind}}}}
\def\id{{{\operatorname{id}}}}
\def\supp{{{\operatorname{supp}}}}
\def\co{{{\operatorname{co}}}}
\def\scoe{{{\operatorname{scoe}}}}
\def\coe{{{\operatorname{coe}}}}
\def\I{{\mathcal{I}}}
\def\Span{{{\operatorname{Span}}}}
\def\event{{{\operatorname{event}}}}
\def\S{\mathcal{S}}

\def\calK{\mathcal{K}}
\def\calP{\mathcal{P}}

\def\whatPI{\widehat{\mathcal{P}}_\infty}
\section{Introduction}

The present paper has two main results.
The first one shows that the reciprocal dual algebra of a unital Kirchberg algebra 
with finitely generated $\K$-groups has a universal property with respect to 
a generating subalgebra and partial isometries.
The second one shows that any Kirchberg algebra with finitely generated $\K$-groups 
has an ergodic automorphism having a unique invariant state.  
The unique invariant state is pure.
The uniqueness of invariant states 
is proved by using the universal relations of generators of the reciprocal Kirchberg algebra 
proved as the first main result.

 The notion of the reciprocality in Kirchberg algebras with finitely generated $\K$-groups 
 has been introduced 
 by the second named author in \cite{Sogabe2022} 
 related to the study of the homotopy groups of the automorphism groups of 
 Kirchberg algebras and bundles of $C^*$-algebras.
 Two unital Kirchberg algebras $\A$ and $\B$ are said to be {\it reciprocal}\/ if 
 $\A$ is $\sqK$-equivalent to $D(C_\B)$ and $\B$ is $\sqK$-equivalent to $D(C_\A)$, where 
 $C_\A$ for a unital Kirchberg algebra $\A$ is defined by the mapping cone algebra 
 for the untal embedding $u_\A: \mathbb{C}\rightarrow \A$, 
 and $D(C_\A)$ is the Spanier--Whitehead $\K$-dual 
 of $C_\A$ (\cite{Sogabe2022}, cf. \cite{KS}, \cite{KP}).
 In this case, we say that $\A$ (resp. $\B$) is reciprocal to $\B$ (resp. $\A$).  
 The reciprocality is a duality between $\K_i(\, - \, )$ and $\Exts^{i+1}(\, - \, )$
  in  unital Kirchberg algebras with finitely generated $\K$-groups
 such that $\A$ is reciprocal to $\B$ if and only if 
 $\K_i(\A) = \Exts^{i+1}(\B), \K_i(\B) = \Exts^{i+1}(\A), i=0,1$,
 where $\Exts^i$ is the strong extension group (\cite[Proposition 3.7]{MatSogabe2}).
 If $\B$ is reciprocal to $\A$ (and hence $\A$ is reciprocal to $\B$), 
 then $\B$ is written as $\whatA$ 
 and called the reciprocal dual of $\A$ or the reciprocal algebra of $\A$.
 It was proved in \cite{Sogabe2022} that $\whatA$ exists  uniquely up to isomorphism of $C^*$-algebras,
 and satisfies $\widehat{\whatA} \cong \A$.
 In \cite{MatSogabe3}, the reciprocal dual $\whatOA$ of a simple Cuntz--Krieger algebra $\OA$
 has been studied so that it may be realized as a unital simple Exel--Laca algebra.
 Hence $\whatOA$ has a universal and uniqueness property 
 in  generating partial isometries subject to 
 its defining matrix with entries in $\{0,1 \}.$

In the first part of the present paper, 
we will generalize the realization result in \cite{MatSogabe3} of the reciprocal dual
$\whatOA$ as a universal $C^*$-algebra 
to more general unital Kirchberg algebras with finitely generated 
$\K$-groups.  
We will first prove the following theorem.

\begin{theorem} 
\label{thm:mainuniversalA3}
Let $\A$ be a unital Kirchberg algebra with finitely generated $\K$-groups.
\begin{enumerate}
\renewcommand{\theenumi}{(\roman{enumi})}
\renewcommand{\labelenumi}{\textup{\theenumi}}
\item
The reciprocal dual $\widehat{\A}$ of $\A$ is the universal unique unital Kirchberg algebra 
written $C^*_{\text{univ}}(\T, \{t_j\}_{j \in \N})$
generated by a separable unital nuclear UCT $C^*$-algebra $\T$ and an infinite family $t_i, \, i \in \N$
of partial isometries satisfying the following properties:
{\begin{enumerate}
\renewcommand{\theenumi}{(\arabic{enumi})}
\renewcommand{\labelenumi}{\textup{\theenumi}}
\item $\T$ contains a $C^*$-subalgebra isomorphic to the $C^*$-algebra $\calK$ 
of compact operators on a separable infinite dimensional 
Hilbert space as an essential ideal
with a minimal projection $e \in \calK \subset \T$ satisfying
\begin{equation} \label{eq:1.1}
(\K_0(\T), [e]_0, \K_1(\T)) \cong
(\Exts^1(\A), [\iota_\A(1)]_s, \Exts^0(\A)),
\end{equation} 
where $[\iota_\A(1)]_s$ is the class of an extension $\iota_\A(1)$ of $\A$
satisfying $\Exts^1(\A)/\Z [\iota_\A(1)]_s =\Extw^1(\A)$ (see Section \ref{sect:Reviewreciprocal} for the definition of $\iota_\A$).
\item  
the partial isometries 
$t_j, j \in \N$ 
satisfy the relations:
\begin{equation}\label{eq:main}
  \begin{cases} & \, \,  t_1^* t_1 = 1_{\whatA}, \qquad 
 t_j^* t_j = 1_\T \quad  \text{ for } j \ge 2,\\
 & \, \,  t_1 t_1^* =e, \qquad 
 1_\T  + \sum_{j=2}^m t_j t_j^* < 1_{\widehat{\A}} \quad \text{ for } m \ge 2, 
\end{cases}
\end{equation} 
where $1_{\widehat{\A}}, 1_\T$ denote the units of $\widehat{\A}, \T$, respectively.
\end{enumerate} 
}
\item
$(\K_0(\T), [e]_0, \K_1(\T))
\cong
(\K_0(\whatA), [1_{\whatA}]_0, \K_1(\whatA)) 
$
and hence the $C^*$-algebra $\whatA$ 
does not depend on the choice of the separable unital nuclear UCT $C^*$-algebra $\T$ 
as long as it satisfies \eqref{eq:1.1}.
\end{enumerate} 
\end{theorem}
Let $\PI$ be the unital Kirchberg algebra such that $\K_0(\PI) = 0,
\K_1(\PI) = \Z$.
Its reciprocal dual $\widehat{\PI}$ 
is also a Kirchberg algebra such that 
$\K_0(\widehat{\PI}) = \Z \oplus \Z$
and $\K_1(\widehat{\PI}) = 0$.
By using the above theorem we may concretely realize the algebra $\widehat{\PI}$ 
as a unital simple Exel--Laca algebra of some infinite matrix $\widehat{P}_\infty$,
which will be studied in Section \ref{sec:Examples}.

The second main result of the present paper shows the existence of
a uniquely ergodic automorphism of a unital Kirchberg algebra with finitely generated 
$\K$-groups.
We will prove the unique ergodicity of an automorphism of unital Kirchberg algebras 
with finitely generated $\K$-groups 
by using the above description of generators as universal $C^*$-algebras. 
There are many important and fundamental results on studies of automorphisms on Kirchberg algebras
such as (cf. Nakamura \cite{Nakamura}, Izumi \cite{Izumi}, Gabe--Szabo \cite{GS}, etc. ). 
By using an idea appeared in the structural analysis of the reciprocal Cuntz--Krieger algebras 
in \cite{MatSogabe3}
the authors recently showed in \cite{MatSogabe4} that 
any unital Kirchberg algebra has an aperiodic ergodic automorphism which has an invariant state, 
where an automorphism $\alpha$ on $\A$ is said to be aperiodic if 
$\alpha^n$ is outer for every $n \in \Z$ with $n\ne 0$.
It is said to be ergodic if the fixed point algebra $\A^\alpha$ of $\A$ under $\alpha$ is 
the scalar multiples of the unit $1_\A$ of $\A$.  
Since any unital Kirchberg algebra $\A$ with finitely generated $\K$-groups
is isomorphic to the reciprocal dual of $\whatA$ (i.e., $\A\cong \widehat{\whatA}$),
$\A$ itself satisfies the universal property of Theorem \ref{thm:mainuniversalA3}.
In the present paper, 
we will give an description of the aperiodic ergodic automorphism given in \cite{MatSogabe4}
in terms of the generating elements appeared  in Theorem \ref{thm:mainuniversalA3}. 
By using the description of the ergodic automorphism, we may prove that 
the invariant stste is unique and pure, 
so that we have the following theorem. 
\begin{theorem}\label{thm:maininvariantstate}
Let $\A$ be an arbitrary unital Kirchberg algebra with finitely generated $\K$-groups.
There exists an aperiodic ergodic automorphism on $\A$ having a unique invariant state.
The invariant state is pure.
\end{theorem}

An automorphism $\alpha$ on a $C^*$-algebra $\A$ is said to be {\it uniquely ergodic}\/ if
$\alpha$-invariant state uniquely exists (cf. \cite{AbadieDykema}).
The above theorem shows that any unital Kirchberg algebra
with finitely generated $\K$-groups has a uniquely ergodic automorphism.
To the best knowledge of authors, 
there seems to be no other previous examples of uniquely ergodic automorphisms
  even on Cuntz--Krieger algebras, except the shifts on the canonical generators on $\OI$.
\begin{corollary}\label{coro:uniquelyergodicCK}
There exists an aperiodic ergodic automorphism on a simple Cuntz--Krieger algebra $\OA$
having a unique invariant state. 
\end{corollary}
\begin{remark}
In particular,
there exists an aperiodic ergodic automorphism on the Cuntz algebra $\O_2$
having a unique invariant state. 
 In \cite{CE}, A. L.  Carey and D. E. Evans constructed  ergodic automorphisms on $\O_2$
 having two distinct invariant states.
 Our  aperiodic ergodic automorphism on the Cuntz algebra $\O_2$
 has only one invariant state so that 
 our automorphism is not conjugate to the ergodic automorphisms constructed by 
 A. L. Carey and D. E. Evans (although they are cocycle conjugate by H. Nakamura's theorem \cite{Nakamura}).
Another easy example of an ergodic automorphism is an automorphism shifting tensor component of $\mathcal{O}_2\cong\bigotimes_{\mathbb{Z}}\mathcal{O}_2$.
This automorphism has infinitely many invariant states given by the product states (i.e., $\varphi^{\otimes\infty} : \mathcal{O}_2^{\otimes\infty}\to\mathbb{C}$), and is not conjugate to our automorphism.
\end{remark}

Throughout the paper, Kirchberg algebras mean 
separable unital simple purely infinite nuclear $C^*$-algebras 
satisfying the UCT.
The UCT $C^*$-algebras mean $C^*$-algebras belonging to the bootstrap class to which the  UCT applies.
We denote by $\Zp$ and $\N$ 
the set of nonnegative integers
and the set of positive integers,
respectively.  

\section{Review of the construction of the reciprocal dual $\widehat{\A}$}\label{sect:Reviewreciprocal}
In this section, 
we will review the construction of  the reciprocal dual $\widehat{\A}$
of a unital Kirchberg algebra $\A$
following \cite{MatSogabe3}.
Throughout the section, 
we denote by $\calK(H)$ the $C^*$-algebra of compact operators
on a separable infinite dimensional Hilbert space $H$, shortly written as $\calK$. 
We refer to the Blackadar's text book \cite{Blackadar} for the basic facts and notations 
related to $\sqK$-groups.

We write $\S = C_0(0,1)$ and $\S^k = \S^{\otimes k}$.
For a separable unital nuclear $C^*$-algebra $\A$,
the mapping cone algebra $C_\A$ is defined by
$C_\A =\{a(t) \in C_0(0,1] \otimes \A\mid a(1) \in {\mathbb{C}}1_{\A}\}$.
Following G. Skandalis \cite{Skandalis},
  we write
\begin{equation*}
\Exts^i(\A) = \sqK(C_\A, \S^{i+1}), \qquad i=0,1.
\end{equation*}
As in \cite{Blackadar}, the group $\Extw^i(\A)$
is given by $\sqK(\A, \S^i)$  for $i=0,1$.
It is well-known that the group $\Exts^1(\A)$ (resp. $\Extw^1(\A)$) 
is realized as the strong (resp. weak) extension group $\Exts(\A)$ (resp. $\Extw(\A)$)
which is the abelian group of strong (resp. weak)  unitary equivalence classes of 
unital Busby invariants (cf. \cite{Blackadar}, \cite{Skandalis}). 
We write $\S\otimes \A$ as $\S \A$.
Applying the $\K$-homology functor $\sqK(\, - \, , \mathbb{C})$ 
to the natural short exact sequence
$0\rightarrow \S\A \rightarrow C_\A \rightarrow \mathbb{C} \rightarrow 0$,
we have a cyclic six term exact sequence of $\Ext_*$-groups:
\begin{equation}\label{eq:cyclic6temExtsw}
\begin{CD}
\Exts^0(\A) @>>> \Extw^0(\A) @>>> \Z \\
@AAA @. @VV{\iota_\A}V \\
0 @<<< \Extw^1(\A) @<<<\Exts^1(\A).
\end{CD}
\end{equation}
The class of $\iota_\A(1)$ in $\Exts^1(\A)$ is denoted by $[\iota_\A(1)]_s$
which corresponds to the class of strong unitary equivalence class 
of a Busby invariant 
${\rm Ad}u\circ \rho$ with a unital trivial extension 
$\rho$ and a unitary $u$ in the Calkin algebra $\B(H)/\calK(H)$
with the Fredholm index $1$.

Throughout the rest of this section,
we fix a Kirchberg algebra $\A$. 
Let $\F$ be a separable unital nuclear UCT $C^*$-algebra containing  
$\calK(H)$ as an essential ideal of $\F$ such that 
\begin{equation}\label{eq:KFExtsA}
(\K_0(\F), [f]_0, \K_1(\F))
\cong
(\Exts^1(\A), [\iota_\A(1)]_s, \Exts^0(\A))
\end{equation}
where $[f]_0$ is the class in $\K_0(\F)$
of a fixed minimal projection $f\in \calK(H)$ so that 
$f \in \calK(H) \subset \F$.
Such a separable unital $C^*$-algebra $\F$ for a given unital Kirchberg algebra $\A$
always exists by \cite[Theorem 1.2]{PennigSogabe}.  
Since $\calK(H)$ is an essential ideal of $\F$,
there exists an injective $*$-homomorphism 
$\pi_\F:\F\rightarrow \B(H)$. 
Let us denote by $e_k, k \in \N$
the canonical basis of $\ell^2(\N)$ defined by
$e_k(m) = \delta_{k,m}, \, k,m \in \N.$ 
Define the Hilbert $C^*$-bimodule $H_\F$ over $\F$
by setting
\begin{equation*}
H_\F: =H\otimes_{\mathbb{C}} \ell^2(\N)\otimes_{\mathbb{C}} \F
\end{equation*}
where $\F$-valued inner product $< \,\,\,\, | \,\,\,\,>_\F$
on $H_\F$ is defined by 
$$
< \xi\otimes e_k \otimes  x \, \, | \, \,  \eta\otimes e_l\otimes y>_\F 
=< \xi \, | \, \eta>_H \delta_{k,l}  x^* y \in \F 
$$
for $\xi, \eta \in H, \, \, e_k, e_l \in \ell^2(\N), \, \, 
 x, y \in \F.$
 The left action of $\F$ on $H_\F$ to the algebra $\L(H_\F)$ of adjointable bounded module maps 
on $H_\F$ is given by 
$$
\varphi_\F:= \pi_\F\otimes 1_{\ell^2(\N)} \otimes 1_\F: 
\F\otimes \mathbb{C} 1_{\ell^2(\N)}\otimes \mathbb{C}1_\F \rightarrow \L(H_\F). 
$$
Take  Pimsner's Toeplitz algebra $\T_{H_\F}$ (see \cite{Pimsner}) 
 which is 
$\sqK$-equivalent to $\F$ via the unital embedding 
$ \F \hookrightarrow \T_{H_\F}.$ 
Since $\F\otimes \mathbb{C} 1_{\ell^2(\N)}\otimes \mathbb{C}1_\F
 \cap \calK(H_\F) = \{ 0 \},$
Kumjian's results \cite[Proposition 2.1]{Kumjian}
and \cite[Theorem 3.1]{Kumjian} together with
\cite[Corollary 4.5]{Pimsner} tell us that 
the Cuntz--Pimsner algebra $\O_{H_\F}$ is a separable unital simple purely infinite
nuclear UCT $C^*$-algebra and hence a Kirchberg algebra such that   
$\O_{H_\F}$ and $\T_{H_\F}$ are canonically isomorphic, 
and the natural unital embedding $\F\hookrightarrow \T_{H_\F}$
yields a $\sqK$-equivalence.
We identify the minimal projection $f$ in $\calK(H)$ 
with the projection in $\OHF$ through the embedding and 
the identification:
\begin{equation}\label{eq:eeA}
f\in \calK(H) \subset \F \hookrightarrow \T_{H_\F} \cong \O_{H_\F} \ni f.
\end{equation} 
We then have
\begin{equation*}
(\K_0(\O_{H_\F}), [f]_0, \K_1(\O_{H_\F}))
\cong  
(\K_0(\T_{H_\F}), [f]_0, \K_1(\T_{H_\F})) 
\cong  
(\K_0(\F), [f]_0, \K_1(\F)) 
\end{equation*}
The reciprocal duality (cf. \cite[Proposition 3.7]{MatSogabe2}) tells us
\begin{equation*}
(\Exts^1(\A), [\iota_\A(1)]_s, \Exts^0(\A)) 
\cong
(\K_0(\widehat{\A}), [1_{\widehat{\A}}]_0, \K_1(\widehat{\A})),
\end{equation*}
so that by the hypothesis \eqref{eq:KFExtsA}, we obtain that 
\begin{equation*}
(\K_0(\O_{H_\F}), [f]_0, \K_1(\O_{H_\F}))
\cong
(\K_0(\widehat{\A}), [1_{\widehat{\A}}]_0, \K_1(\widehat{\A})),
\end{equation*}
showing that 
\begin{equation}\label{eq:constructionAhat}
f \O_{H_\F} f \cong \widehat{\A},
\end{equation}
by the Kirchberg--Philips's classification theorem of Kirchberg algebras 
(\cite{Kirchberg}, \cite{Phillips}).

\section{Universality of the Cuntz--Pimsner algebra $\OHF$}
Let $\F$ be a separable  $C^*$-algebra with unit $1_\F$ 
containing $\calK(H)$ as an essential ideal.
Take a minimal projection $f$ in $\calK(H)$.
In this section, we begin with only $f \in \calK(H) \triangleleft \F$
by which one may construct the Cuntz--Pimsner algebra $\OHF$.
We need not provide a Kirchberg algebra $\A$ in this section,
so we do not assume the hypothesis \eqref{eq:KFExtsA}. 
\begin{lemma}\label{lem:UniversalAhat1}
The $C^*$-algebra $\O_{H_\F}$ is generated by $\F$ and isometries $S_i, i \in \N$
satisfying
\begin{equation}\label{eq:UniversalAhat1}
 S_i^* S_i = 1_{\O_{H_\F}}, \quad i \in \N
 \quad \text{ and } \quad
\sum_{j=1}^m S_j S_j^* < f, \quad m \in \N.
\end{equation}
\end{lemma}
\begin{proof}
We follow the previous construction of the $C^*$-algebra $\OHF$ in detail.
Take a unit vector $\xi_0 \in H$ such that $ f H= \mathbb{C}\xi_0$.
The Hilbert $C^*$-bimodule $H_\F = H \otimes_{\mathbb{C}} \ell^2(\N) \otimes_{\mathbb{C}} \F$ 
has a left action $\varphi_\F: \F \rightarrow \L(H_\F)$
defined by $\varphi_\F(a) = \pi_\F(a) \otimes 1_{\ell^2(\N)} \otimes 1_\F$ for $a \in \F$.
We write $\pi_\F(a)$ as $a$ for short.
The action $\varphi_\F$ is denoted by $\varphi$.
Consider the Fock space $\E_\F$ of $H_\F$ defined by
\begin{equation*}
\E_\F :=\F \oplus \bigoplus_{n=1}^\infty H_\F^{\otimes n}
\end{equation*}
where 
$H_\F^{\otimes n}$ is the $n$-times relative tensor product 
$\overbrace{H_\F \otimes_\F H_\F \otimes _\F \cdots \otimes_\F H_F}^{\text{$n$ times}}$ 
over $\F$.
The creation operator 
$T_{\xi\otimes e_k \otimes y}$ for $\xi \otimes e_k \otimes y \in H_\F$ 
is defined by 
\begin{align*}
& T_{\xi\otimes e_k \otimes y} (a) = \xi\otimes e_k \otimes ya, \quad a \in \F,\\
& T_{\xi\otimes e_k \otimes y} (\zeta_1\otimes \zeta_2\otimes \cdots \otimes \zeta_n)
 = (\xi\otimes e_k \otimes y)\otimes \zeta_1\otimes \zeta_2\otimes \cdots \otimes \zeta_n 
 \in H_\F^{\otimes n+1}
\end{align*}
for $\zeta_1\otimes \zeta_2\otimes \cdots \otimes \zeta_n \in H_\F^{\otimes n}$.
The Toeplitz algebra $\T_{H_\F}$ for the $C^*$-bimodule $H_\F$ 
is defined by the $C^*$-algebra 
$C^*(T_{\xi\otimes e_k \otimes y}\mid \xi\in H_\F,\, k \in \N,\, y \in \F)$
generated by the creation operators
$T_{\xi\otimes e_k \otimes y} \in \L(\E_\F)$
for $\xi\otimes e_k \otimes y \in H_\F$.
Define $\varphi_+: x \in \F \rightarrow \varphi_+(x) \in \L(\E_\F)$ by
\begin{align*}
& \varphi_+(x) (a) = xa, \quad a \in \F, \\ 
&\varphi_+(x)(\zeta_1\otimes \zeta_2\otimes \cdots \otimes \zeta_n)
=(\varphi(x)\zeta_1)\otimes \zeta_2\otimes \cdots \otimes \zeta_n,
\end{align*}
$\zeta_1\otimes \zeta_2\otimes \cdots \otimes \zeta_n \in H_\F^{\otimes n}.$
We then have
\begin{equation}\label{eq:Txi0ek1*Txi0ek1}
T_{\xi\otimes e_k \otimes y}^* T_{\eta\otimes e_l \otimes z}
=  \varphi_+(<\xi\otimes e_k \otimes y \mid \eta\otimes e_l \otimes z >)
=  \varphi_+(<\xi\mid \eta>_H \delta_{k,l} y^*z)
\end{equation}
for $\xi\otimes e_k \otimes y, \, \eta\otimes e_l \otimes z \in H_\F$,
and
\begin{equation*}
 \varphi_+(x) T_{\xi_0\otimes e_k \otimes 1_\F} \varphi_+(y)
 =T_{\varphi(x) \xi_0\otimes e_k \otimes y}.
\end{equation*}
Hence the $C^*$-algebra $\T_{H_\F}$
is generated by 
$\varphi_+(x), x \in \F$ and $T_{\xi_0\otimes e_k\otimes 1_\F}, k \in \N$.
We write 
\begin{equation}\label{eq:defSk}
S_k := T_{\xi_0\otimes e_k\otimes 1_\F}
\end{equation}for $k \in \N$.
By \eqref{eq:Txi0ek1*Txi0ek1}, we have
$S_k^* S_l = \delta_{k,l}.$
It is easy to see that 
\begin{equation*}
S_k S_k^* (\zeta_1\otimes \zeta_2\otimes \cdots \otimes \zeta_n)
=  
 {\begin{cases}
\zeta_1\otimes \zeta_2\otimes \cdots \otimes \zeta_n 
& \text{ if } \zeta_1 \in \mathbb{C}\xi_0\otimes \mathbb{C}e_k \otimes \F, \\
0 & \text{ otherwise,}
\end{cases} }
\end{equation*}
so that $S_k S_k^*$ is the projection onto the subspace spanned by
\begin{equation*}
 (\xi_0\otimes e_k \otimes \F) \oplus \bigoplus_{n=1}^\infty (\xi_0\otimes e_k \otimes \F)\otimes H_\F^{\otimes n}.
\end{equation*}
The projection onto the subspace
\begin{equation*}
f \F \oplus  (\xi_0\otimes \ell^2(\N) \otimes \F) 
\oplus \bigoplus_{n=1}^\infty (\xi_0\otimes \ell^2(\N) \otimes \F)\otimes H_\F^{\otimes n}
\end{equation*}
corresponds to the minimal projection $f$ in $\calK(H)$
by identifying $f$ with $\varphi_+(f).$
We may regard $f$ as an element of $\OHF$
through the correspondence \eqref{eq:eeA}
 under the natural identification between $\T_{H_\F}$ and $\OHF$. 
Since $S_k S_k^* < f$ for all $k \in \N$ and $\{S_kS_k^*\}_{k\in\mathbb{N}}$ are mutually orthogonal, we have
\begin{equation*}
\sum_{j=1}^m S_j S_j^* < f \quad \text{ for all } m \in \N.
\end{equation*}
We may write $\varphi_+(x) T_{\xi\otimes e_k \otimes 1}\varphi_+(y)$ as 
$x S_k y$ for short, so that the $C^*$-algebra $\OHF$ is generated by $\F$ and $S_k, k\in \Z$ 
satisfying \eqref{eq:UniversalAhat1}.
\end{proof}
We note that the algebraic structure of the Cuntz--Pimsner algebra $\OHF$
is determind by the following two relations (cf. \cite[Definition 4.6.14]{BO}):
\begin{equation*}
\begin{cases}
& \bullet \quad T_{\xi\otimes e_k \otimes y}^* T_{\eta\otimes e_l \otimes z}
=  <\xi\mid \eta>_H \delta_{k,l} \varphi_+( y^*z), \\
& \bullet \quad \varphi_+(x) T_{\xi\otimes e_k \otimes 1}\varphi_+(y) = T_{x\xi\otimes e_k \otimes y}
\end{cases}
\end{equation*}
for $\xi, \eta \in H_\F, \, k,l \in \N, \, x,y \in \F.$
\begin{lemma}\label{lem:universalityforHF}
Keep the unital $C^*$-algebra $\F$ containing $\calK(H)$ as an essential ideal,
and fix a minimal projection $f \in \calK(H)$.
Suppose that there exist a unital representation
$\pi: \F \rightarrow \B$ of $\F$ to a unital $C^*$-algebra $\B$
and a family of isometries $s_j \in \B, j \in \N$ satisfying
\begin{equation}\label{eq:sjB}
 s_i^* s_i = 1_\B \quad \text{ for } i\in \N 
 \quad \text{ and } \quad 
\sum_{j=1}^m s_j s_j^* < \pi(f) \quad \text{ for } m\in \N.
\end{equation}
Then there exists a linear map $\tau: H_\F\rightarrow \B$ such that 
\begin{equation}\label{eq:tauB}
\begin{cases}
& \bullet \quad \tau(a u b) = \pi(a) \tau(u) \pi(b)  \quad \text{ for } u \in H_\F \, \, \text{ and } \, a,b \in \F, \\
& \bullet \quad \tau(u)^* \tau(v) = \pi(<u\mid v>_\F) \quad \text{ for } u, v \in H_\F.
\end{cases}
\end{equation}
Hence the pair $(\pi,\tau)$ gives rise to a covariant representation of the Hilbert $C^*$-bimodule
$H_\F$ such that 
the $C^*$-algebra $C^*(\tau(u)\mid u \in H_\F)$ generated by 
$\tau(u), u \in H_\F$ coincides with 
the $C^*$-algebra $C^*(\pi(\F), \{ s_j\}_{j \in \N})$ generated by 
$\pi(\F)$ and $\{s_j\}_{j \in \N}.$
\end{lemma}
\begin{proof}
Take a unit vector $\xi_0 \in H$ such that 
$f H = \mathbb{C} \xi_0.$
For $\xi\in H_\F$, let us denote by 
$t_{\xi, \xi_0}\in \calK(H)$ the operator on $H$ defined by
$t_{\xi,\xi_0}(\zeta) = \xi<\xi_0\mid \zeta>_H, \, \zeta \in H_\F.$
As $\calK(H) \subset \F,$
it belongs to $\F.$
For $u = \xi\otimes e_k\otimes y \in H_\F,$ define
\begin{equation*}
\tau(u) = \pi(t_{\xi,\xi_0}) s_k \pi(y),
\end{equation*}
and extend it linearly to $H_\F$
in a natural way.
As 
$t_{\xi_0, \xi_0} =f$ and $s_k s_k^* < \pi(f),$
we have for $a,b \in \F$, 
$u = \xi \otimes e_k \otimes y, \, v = \eta \otimes e_l\otimes z \in H_\F$,
\begin{equation*}
\tau(aub)
= \pi(t_{a\xi, \xi_0}) s_k \pi(yb) =  \pi(a) \pi(t_{\xi, \xi_0}) s_k \pi(y) \pi(b) 
=  \pi(a) \tau(u)  \pi(b) \
\end{equation*}
and
\begin{equation*}
\tau(u)^* \tau(v)
=  \pi(y)^* s_k^* \pi(<\xi \mid \eta>_H e) s_l \pi(z)  
=  \pi(<u \mid v>_{H_\F}).
\end{equation*}
As 
$\varphi(\F) \cap \calK(H_\F) = \{ 0 \}$,
\cite[Proposition 2.1]{Kumjian}
shows that the representation
$(\pi,\tau)$ is covariant by  \cite[Definition 4.6.9]{BO}.
For $k \in \N$ and $a \in \F$, we have
$s_k = \tau(\xi_0\otimes e_k \otimes 1)$ and
\begin{align*}
\pi(a) 
= &
\pi(a) s_k^* s_k 
= \tau(a<\xi_0\otimes e_k \otimes 1_{\ell^2(\N)} \mid \xi_0\otimes e_k \otimes 1_{\ell^2(\N)} >_\F) \\
=& \tau(<\xi_0\otimes e_k \otimes a^* \mid \xi_0\otimes e_k \otimes 1 >_\F),
\end{align*}
showing that 
$C^*(\tau(u)\mid u \in H_\F)=
C^*(\pi(\F), \{ s_j\}_{j \in \N}).$ 
\end{proof}
Therefore we have the following proposition.
\begin{proposition}\label{prop:universalityforOHF}
Keep the unital $C^*$-algebra $\F$ containing $\calK(H)$ as an essential ideal,
and fix a minimal projection $f \in \calK(H)$.
Suppose that there exist a unital representation
$\pi: \F \rightarrow \B$ of $\F$ to a unital $C^*$-algebra $\B$
and a family of isometries $s_j \in \B, j \in \N$ 
satisfying \eqref{eq:sjB}.
Then there exists a surjective $*$-isomorphism
$\Phi: \OHF \rightarrow C^*(\pi(\F), \{s_j\}_{j\in \N})$
such that 
\begin{equation}\label{eq:3.7}
\Phi(S_i) = s_i \quad \text{ for } i \in \N
\quad \text{ and } \quad 
\Phi(\varphi_+(x)) =\pi(x) \quad \text{ for } x \in \F.
\end{equation}
Hence we have $\OHF \cong C^*(\pi(\F), \{s_j\}_{j\in \N}).$
\end{proposition}
\begin{proof}
As in \cite[Definition 4.6.14]{BO},
the Cuntz--Pimsner algebra $\OHF$ has the universality
for a covariant representation
such that there exists 
 a surjective $*$-homomorphism
$\Phi: \OHF \rightarrow C^*(\tau (u)\mid u\in H_\mathcal{F})=C^*(\pi(\F), \{s_j\}_{j\in \N})$
satisfying 
\begin{equation*}
\Phi(T_u) = \tau(u) \quad \text{ for } u \in H_\F
\quad \text{ and } \quad 
\Phi(\varphi_+(x)) =\pi(x) \quad \text{ for } x \in \F.
\end{equation*}
As 
$
\tau(\xi_0\otimes e_k \otimes 1)
=\pi(t_{\xi_0,\xi_0}) s_k \pi(1)
= s_k,
$
we obtain \eqref{eq:3.7}.
 As in \cite[Theorem 2.8]{Kumjian}, the $C^*$-algebra
$\OHF$ is simple, 
so that $\Phi: \OHF \rightarrow C^*(\pi(\F), \{s_j\}_{j\in \N})$
is isomorphic.
\end{proof}
\begin{remark}\label{rem:faithfulness}
(i)
We have to remark that a unital representation $\pi:\F\rightarrow \B$ 
in Proposition \ref{prop:universalityforOHF}
 with partial isometries $\{ s_j \}_{j \in \N}$
 satisfying the relations \eqref{eq:sjB}
 automatically becomes faithful, because
 $\F$ contains $\calK(H)$ as an essential ideal and $\pi(f)$
 for  the minimal projection $f$ does not vanish 
 by \eqref{eq:sjB}.
 One in fact sees that if $\ker(\pi) \ne \{0 \},$ then
 $\ker(\pi) \cap \calK(H) \ne \{0 \}$ and hence 
 $\ker(\pi) \supset \calK(H)$ so that $f \in \ker(\pi)$,
 a contradiction to the relations \eqref{eq:sjB}.
 
 (ii) If in particular $\F$ is a nuclear UCT $C^*$-algebra, 
 so is the $C^*$-algebra $\OHF$ by \cite[Theorem 3.1]{Kumjian},
 so that $\OHF$ is a Kirchberg algebra.
 \end{remark}  
The proposition above says the following theorem.
\begin{theorem}\label{thm:universalityforOHF}
Let $\F$ be a separable unital $C^*$-algebra containing $\calK(H)$ as an essential ideal,
and  $f \in \calK(H)$ a fixed minimal projection.
 The $C^*$-algebra $\OHF$ is the universal unique unital $C^*$-algebra
 $C^*_{\text{univ}}(\F, \{S_j\}_{j \in \N})$ generated by 
 $\F$ and a family $\{S_j\}_{j \in \N} $ of isometries subject to the relations:
 \begin{equation}\label{eq:SiB}
S_i^* S_i = 1_\F \quad \text{ for } i\in \N 
\quad \text{ and } \quad
\sum_{j=1}^m S_j S_j^* < f \quad \text{ for } m\in \N.
\end{equation}
\end{theorem}

\section{Universality of the reciprocal algebra $\whatA$}
Keep the situation 
that $\F$ is a separable unital $C^*$-algebra containing $\calK(H)$ 
as an essential ideal and fix a minimal projection $f \in \calK(H)$.
Let us represent $\OHF$ to be the $C^*$-algebra
$C^*_{\text{univ}}(\F, \{S_j\}_{j \in \N})$ as in Theorem \ref{thm:universalityforOHF}.
The unit $1_\F$ of $\F$ coincides with the unit $1_\OHF$ of $\OHF$.
We define a $C^*$-algebra  $\T$  and a projection $e \in \T$ by
\begin{equation} \label{eq:defT}
\T := S_1 \F S_1^* \qquad 
e:= S_1 f S_1^*,
\end{equation}
and a family $\{ T_j \}_{j \in \N}$
of partial isometries in $C^*_{\text{univ}}(\F, \{S_j\}_{j \in \N})$ by
\begin{equation}\label{eq:defTj}
T_1 := S_1 f 
\quad
\text{ and }
\quad
T_j := 
S_j S_1^*  \quad \text{for} \quad j\ge 2.
\end{equation}
Hence the unit $1_\T$ of the algebra $\T$ is $S_1 S_1^*$.
Note that $\mathcal{T}\not\subset\mathcal{F}$.
\begin{lemma}\label{lem:relationsTjeV}
Put
$
V: = S_1 \,\, (\not\in C^*(\T, \{T_j\}_{j \in \N})). 
$
\begin{enumerate}
\renewcommand{\theenumi}{(\roman{enumi})}
\renewcommand{\labelenumi}{\textup{\theenumi}}
\item
The following identities (T1), (T2), (T3) and (T4) hold:

$(T1)\quad  T_j^* T_j = 1_\T \quad j\ge 2.$

$(T2) \quad T_1 T_1^* = e, \qquad 1_\T + \sum_{j=2}^m T_j T_j^* < T_1^* T_1 $ for  $m \ge 2$.

$(T3) \quad T_1 T_1^*= V T_1^*, \qquad T_1^* T_1= V^* T_1.$ 

$(T4) \quad V V^* = 1_\T, \qquad V^*V = 1_{\OHF}.$
\item
The $C^*$-algebra $\OHF$ is generated by $\T, $  $\{T_j\}_{j \in \N}$ and $V$, that is,
\begin{equation}\label{eq: OHFTTjV}
\OHF =  C^*(\T, \{T_j\}_{j \in \N}, V).
\end{equation}
\end{enumerate}
\end{lemma}
\begin{proof}
(i)
$(T1): $ For $j \ge 2$, we have 
$
T_j^* T_j = S_1 S_j^* S_j S_1^* = S_1 1_\F S_1^* = 1_\T.
$

$(T2): $ We have
$T_1 T_1^* 
= S_1 f S_1^* = e.
$
For $m \ge 2$, we have
\begin{equation*}
1_\T + \sum_{j=2}^m T_j T_j^* 
=  S_1 S_1^*  + \sum_{j=2}^m S_jS_1^* S_1 S_j^*  
=  \sum_{j=1}^m S_j S_j^* 
< f =  f S_1^* S_1 f = T_1^* T_1. 
\end{equation*}
The identities
$
T_1 T_1^* 
= S_1 f S_1^* 
= V T_1^*, \, \, 
T_1^* T_1
= f = S_1^* S_1 f = V^* T_1, 
$
and
$V V^* =  S_1 S_1^* = S_1 1_\F S_1^* = 1_\T, \, \,
V^* V =  S_1^* S_1 = 1_\F (= 1_{\OHF})
$
show us (T3) and (T4).

(ii)
It suffices to show that 
$
\OHF 
= C^*(S_1\F S_1^*, \, S_1 f, \, \{S_j S_1^*\}_{2\le j \in \N}, \, S_1 ).
$ 
Since $\OHF = C^*(\F, \{S_j \}_{j \in \N})$,
the inclusion relation 
 $\OHF \supset C^*(S_1\F S_1^*, \, S_1 f, \, \{S_j S_1^*\}_{2\le j \in \N}, \, S_1 )$
 is obvious.
 Conversely, for $a \in \F$ and $ 2\le j \in \N$, we have
 $$
 a = S_1^* S_1 a S_1^* S_1= S_1^*\cdot  S_1 a S_1^* \cdot S_1, \qquad
 S_j = S_j S_1^* S_1  = S_j S_1^*\cdot S_1, 
 $$ 
 so that 
$\OHF \subset C^*(S_1\F S_1^*, \, S_1 f, \, \{S_j S_1^*\}_{2\le j \in \N}, \, S_1 ).$
 \end{proof}
\begin{lemma}\label{UniversalhatA2}
The corner $f \O_{H_\F} f$ of the $C^*$-algebra $\OHF$
is generated by 
\begin{equation}\label{eq:generatotsSnF}
S_n \F S_n^*, \quad S_n f, \quad S_n S_i S_n^*, \qquad n, i\in \N.
\end{equation}
\end{lemma}
\begin{proof}
We first note that 
the identity $f = (S_1 f)^* S_1 f$ holds.
Since $\sum_{j=1}^m S_j S_j^* < f$ for all $m \in \N$,
we have $f S_n = S_n$ and $S_n^* f = S_n$ for all $n \in \N.$
Hence it is obvious to see that all the elements written in \eqref{eq:generatotsSnF}
belong to $f \O_{H_\F} f.$
Let us denote by 
$\N^* = \cup_{m=0}^\infty \N^m$ where
$\N^0=\{\emptyset\}, \N^m = \{(\mu_1,\dots, \mu_m) \mid \mu_i \in \N\}$.
For $\mu =(\mu_1,\dots, \mu_m) \in \N^m$, 
we write 
$S_\mu = S_{\mu_1}\cdots S_{\mu_m}$
and $S_\mu =1_{\OHF}$ for $\mu = \emptyset$. 
Since $S_j^*S_i = \delta_{j,i}1_{\OHF}$,
the linear span of elements 
of the form $f a_1 a_2 \cdots a_k f$ for
$a_i \in \F$ or $a_i = S_\mu S_\nu^*$ with $\mu, \nu \in \N^*$
is dense in $f \OHF f$.
As $1_{\OHF} = 1_\F \in \F$,
one may assume that  
the element 
$f a_1 a_2 \cdots a_k f$ 
satisfies  
$a_1, a_3 \in \F$ and $a_2 = S_\mu S_\nu^*$
with $\mu=(\mu_1,\dots, \mu_m), 
\nu = (\nu_1,\dots,\nu_n) \in \N^*$.

We have three cases.
 
Case 1: $\, m \ge 1$ and $n \ge 1.$

As $S_j = f S_j$ and $S_j^* = S_j^* f$,
we have
\begin{align*}
  & f a_1 a_2 a_3 \cdots a_k f \\
= & f \cdot S_1^* S_1\cdot  a_1\cdot  S_1^* S_1 \cdot f S_\mu S_\nu^* f a_3 \cdots a_k f \\
= & (S_1 f)^* S_1 a_1 S_1^*\cdot S_1 f \cdot S_{\mu_1} S_{\mu_2}
\cdots S_{\mu_m}S_{\nu_n}^*\cdots S_{\nu_1}^* f a_3 \cdots a_k f 
\end{align*}
where
\begin{equation*}
   S_{\mu_1} S_{\mu_2}\cdots S_{\mu_m}S_{\nu_n}^*\cdots S_{\nu_1}^*
=  S_{\mu_1} S_{\mu_2}S_{\mu_1}^*\cdot S_{\mu_1}S_{\mu_3} S_{\mu_1}^* 
\cdots S_{\mu_1}S_{\mu_m}S_{\mu_1}^*\cdot ( S_{\mu_1} S_{\nu_n}^*\cdots S_{\nu_1}^* f) 
\end{equation*} 
and
\begin{equation*}
  S_{\mu_1}   S_{\nu_n}^*\cdots S_{\nu_1}^* f
=  S_{\mu_1} S_{\nu_n}^*S_{\mu_1}^*\cdot S_{\mu_1} S_{\nu_{n-1}} S_{\mu_1}^*
   \cdot S_{\mu_1}\cdots S_{\mu_1} S_{\nu_1}^* S_{\mu_1}^* \cdot S_{\mu_1} f.
\end{equation*}
Hence the word 
$f a_1 a_2(= f a_1 f S_{\mu_1} S_{\mu_2}\cdots S_{\mu_m}S_{\nu_n}^*\cdots S_{\nu_1}^*) $
is written in words of elements of the form of \eqref{eq:generatotsSnF}.

Case 2: $\, m=0$ and $n \ge 1.$

We then have
\begin{align*}
  & f a_1 a_2 a_3 \cdots a_k  f \\
= & f a_1  S_\nu^* f a_3 \cdots a_k f \\
= & f \cdot S_1^* S_1\cdot a_1 \cdot S_1^* S_1\cdot S_{\nu_n}^*
\cdot S_1^* S_1 \cdots S_1^* S_1\cdot S_{\nu_1}^*\cdot S_1^* S_1
\cdot f a_3 \cdots a_k f \\
= & (S_1 f)^* \cdot S_1 a_1  S_1^* \cdot (S_1 S_{\nu_n} S_1^*)^* 
 \cdots  (S_1 S_{\nu_1} S_1^*)^* \cdot S_1 f a_3 \cdots a_k f. 
\end{align*}
Hence the word $f a_1 a_2(=f a_1  S_\nu^*)$ 
is written in words of elements of the form of \eqref{eq:generatotsSnF}.

Case 3: $\, m \ge 1$ and $n =0.$

We then have
\begin{align*}
  & f a_1 a_2 a_3 \cdots a_k f \\
= & f S_1^* S_1 a_1  S_1^* S_1 S_{\mu_1} S_1^* S_1 
\cdots S_1^* S_1 S_{\mu_m}S_1^* S_1  a_3 S_1^* S_1 a_4 \cdots a_k f \\
= & (S_1 f)^* \cdot S_1 a_1  S_1^* \cdot S_1 S_{\mu_1} S_1^* \cdot S_1 
\cdots S_1^* \cdot S_1 S_{\mu_m}S_1^* \cdot S_1 a_3 S_1^* \cdot f S_1 a_4 \cdots a_k f. 
\end{align*}
Hence the word $f a_1  S_{\mu_1}\cdots S_{\mu_m} S_1^* S_1 a_3 S_1^*$ 
is written in words of elements of the form of \eqref{eq:generatotsSnF}.

Since $f S_1 f = f\cdot S_1 f$,
$
f S_1 S_j f = f S_1 S_jS_1^*\cdot  S_1 f
$ 
and
$
f S_1 S_j^* f = f S_1 S_j^* S_1^*\cdot  S_1 f,
$ 
by induction on the length $k$ of the word 
$f a_1 a_2 a_3 \cdots a_k f$, 
we may conclude that 
the word 
$
f a_1 a_2 a_3 \cdots a_k f
$
is written in words of elements of the form of \eqref{eq:generatotsSnF}.
 \end{proof}
\begin{lemma}\label{UniversalhatA3}
The $C^*$-algebra $f \O_{H_\F} f$ is generated by 
\begin{equation}\label{eq:generatotsS1F}
S_1 \F S_1^*, \quad S_1 f, \quad S_1 S_i S_1^*, \qquad  i\in \N.
\end{equation}
\end{lemma}
\begin{proof}
We set 
\begin{align*}
\A_1 := & C^*(S_1 \F S_1, \, \, S_1 f,\, \,  S_1 S_i S_1^* \mid i \in \N), \\ 
\A_* := & C^*(S_n \F S_n, \, \,  S_n f, \, \, S_n S_i S_n^* \mid n, i \in \N)
\end{align*}
and
$
\T: = S_1 \F S_1^*, \, \, S:= S_1 f, \, \, R_i :=S_1 S_i S_1^*, \, i \in \N.
$
By the preceding lemma, we know that 
$f \O_{H_\F} f =\A_*$, and $\A_1 \subset \A_*$ is claer.
We will show $\A_1\supset \A_*$.
For $S_i S_1^*$ with $i \in \N$, we have
\begin{equation*}
S_i S_1^* = f S_i S_1^*
= f S_1^* S_1  S_i S_1^*
=S^*\cdot R_i
\end{equation*}
so that 
$S_i S_1^*  \in \A_1$.
For $n \in \N$ and $a \in \F$, as $S_1^* S_1 =1_{\O_{H_\F}}$,  
we have
\begin{equation*}
S_n a S_n^*
= S_n \cdot S_1^*S_1\cdot  a \cdot S_1^*S_1\cdot  S_n^* 
= S^* R_n \cdot S_1  a  S_1^* \cdot R_n^* S 
\end{equation*}
so that 
$S_n a S_n^* \in \A_1$ for all $n \in \N$,
and similarly 
the identity 
$
S_n S_i S_n^*
=  S^* R_n \cdot R_i \cdot R_n^* S 
$ 
holds so that 
$S_n S_i S_n^* \in \A_1$ for all $n, i \in \N$.
We also have
\begin{equation*}
S_n f
 =  S_n \cdot S_1^*S_1 \cdot f 
=  S^* R_n \cdot S 
\end{equation*}
so that 
$S_n f \in \A_1$ for all $n \in \N$, 
showing that  $\A_* \subset \A_1$ and hence $\A_* = \A_1.$
\end{proof}
Recall that 
the $C^*$-subalgebra  $\T$ of $\OHF$ 
and a family $\{ T_j \}_{j \in \N}$
of partial isometries in $\OHF$ are defined by
\eqref{eq:defT} and \eqref{eq:defTj}, respectively.
Hence we have the following lemma.
\begin{lemma}\label{UniversalhatA4}
The $C^*$-algebra $f \O_{H_\F} f$ is generated by 
$\T$ and $T_j, j \in \N$: 
\begin{equation}\label{eq:generatorsTTj}
f \OHF f = C^*(\T, \{T_j\}_{j\in \N}).
\end{equation}
\end{lemma}
\begin{proof}
We have
$
S_1 S_j S_1^* = S_1 f S_j S_1^* = T_1 T_j
$
for $j \ge 2$, 
and 
\begin{equation*}
S_1 S_1 S_1^* = S_1 f S_1 S_1^* =S_1 f S_1 S_2^*  S_2  S_1^* = T_1 T_2^*T_2
\quad\text{ for } j = 1.
\end{equation*}
On the other hand,  we have
$
T_j  
=(S_1 f)^* S_1 S_j S_1^*
$
for $j \ge 2$.
By Lemma \ref{UniversalhatA3}, we conclude that 
$
f \OHF f = C^*(\T, \{T_j\}_{j \in \N}).
$
\end{proof}  
Recall that  
a projection $e \in \T$ is defined by $e= S_1 f S_1^*\in \T$. 
\begin{lemma}\label{UniversalhatA5}
Assume that there exist a degenerate representation
$\pi: \T \rightarrow \B(\fH)$ 
and a family $t_j \in \B(\fH), j \in \N$
of partial isometries on a Hilbert space $\fH$ satisfying
\begin{equation*}
\begin{cases}
& (t1) \quad t_1^* t_1 = 1_{\B(\fH)}, \qquad
      t_j ^* t_j = \pi(1_\T) \quad  \text{ for } j\ge 2,
 \\
& (t2) \quad t_1 t_1^* = \pi(e), \qquad
\pi(1_\T) + \sum_{j=2}^m t_j t_j^* < 1_{\B(\fH)} \, \text{ for } 2 \le m \in \N.
\end{cases}
\end{equation*}
Then  there exist a Hilbert space 
$\tfH = \fH\oplus \fK$ for some Hilbert space $\fK$ and 
an isometry $v \in \B(\tfH)$ such that 
\begin{equation*}
\begin{cases}
& (t3) \quad t_1 t_1^* = v t_1^*, \qquad t_1^* t_1 = v^* t_1, \\
& (t4) \quad  v v^* = \pi(1_\T), \qquad v^* v = 1_{\B(\tfH)}.
\end{cases}
\end{equation*}
where $\B(\fH) = \B(\fH\oplus 0) \subset \B(\tfH)$.
\end{lemma}
\begin{proof}
Put $\fK : = \pi(1_\T - e)\fH$
and
$\tfH = \fH \oplus \fK$.
Represent $C^*(\pi(\T), \{ t_j \}_{j \in \N})$ 
on $\tfH$ by 
$$
x (h, k) := (x(h), 0) \quad 
\text{ for } x \in C^*(\pi(\T), \{t_j\}_{j \in \N}),\;\; (h, k)\in\fH\oplus\fK.
$$
Hence $x \in C^*(\pi(\T), \{t_j\}_{j \in \N})$ 
is written as
$x =
\begin{bmatrix}
x & 0 \\
0 & 0
\end{bmatrix}
\text{ on } \fH \oplus \fK.
$
Define 
\begin{equation*}
v :=
\begin{bmatrix}
t_1 & \pi(1_\T -e) \\
0 & 0
\end{bmatrix}
\text{ on } \fH \oplus \fK.
\end{equation*}
Since 
$$
\pi(1_\T -e) t_1 =\pi(1_\T -e) t_1t_1^* t_1 = \pi(1_\T -e) \pi(e) t_1 =0,
$$
it is direct to see that 
$v$ is an isometry on $\tfH$ satisfying $ (t3)$ and $(t4)$.
\end{proof}
\begin{remark}\label{rem:faithfulness2}
The representation $\pi:\T\rightarrow \B(\fH)$
with the partial isometries $\{t_j \}_{j \in \N}$
satisfying $(t1)$ and $(t2)$
automatically becomes faithful, because 
$\T$ contains a $C^*$-subalgebra $\calK$ as an essential ideal,
and $\pi(e)$ for the minimal projection $e$ in $\calK$
does not vanish by the relations 
$t_1 t_1^*= \pi(e)$ and $t_1^* t_1 = 1_{\B(\fH)}.$ 
\end{remark}
\begin{proposition}\label{prop:UniversalhatA6}
Keep the assumption of Lemma \ref{UniversalhatA5}.
Take the Hilbert space $\tfH$ and the isometry $v$ on $\tfH$
satisfying (t3) and (t4) in Lemma \ref{UniversalhatA5}.
Put
the partial isometries 
$\tilde{s}_j, \, j \in \N$ on $\tfH$
and
a representation $\tilde{\pi}: \F \rightarrow \B(\tfH)$
by setting 
\begin{equation}
\tilde{s}_j
:= \begin{cases}
v^* t_1 v & \text{ for j=1}\\ 
 v^* t_1 t_j v & \text{ for } j \ge 2,
 \end{cases}
\quad  \text{ and } \quad
\tilde{\pi}(a) : = v^* \pi(S_1a S_1^*)v
\quad \text{ for } a \in \F.
\end{equation}
Then we have
\begin{enumerate}
\renewcommand{\theenumi}{(\roman{enumi})}
\renewcommand{\labelenumi}{\textup{\theenumi}}
\item
$\tilde{\pi}: \F \rightarrow \B(\tfH)$
is a unital representation 
such that 
\begin{equation*}
\tilde{s}_i^* \tilde{s}_i = 1_{\B(\tfH)} \quad \text{ for } i \in \N \quad
\text{ and }
\quad
\sum_{j=1}^m \tilde{s}_j \tilde{s}_j< \tilde{\pi}(f) 
 \quad \text{ for } m \in \N.
\end{equation*}
\item
There exists an injective $*$-homomorphism
$\Phi: \OHF\rightarrow \B(\tfH)$ such that 
{
\begin{enumerate}
\renewcommand{\theenumi}{(\arabic{enumi})}
\renewcommand{\labelenumi}{\textup{\theenumi}}
\item 
$
\Phi(S_i) = \tilde{s}_i
\text{  for } 
 i \in \N
$
  and 
$
 \Phi(a) = \tilde{\pi}(a) \text{ for } a \in \F,
$ 
and hence 
\begin{equation*}
\OHF \cong C^*(\tilde{\pi}(\F), \{\tilde{s}_j\}_{j\in \N})
\end{equation*}
\item $\Phi(f \OHF f) = C^*(\pi(\T), \{ t_j \}_{j \in \N})$ and hence
\begin{equation*}
f \OHF f \cong C^*(\pi(\T), \{ t_j \}_{j \in \N}).
\end{equation*}
\end{enumerate}
}
\end{enumerate}
\end{proposition}
\begin{proof}
(i) The equalities 
$$
\tilde{\pi}(1_\F) 
= v^* \pi(S_1 1_\F S_1^*) v 
= v^*\pi(1_\T) v 
= v^* v v^* v = 1_{\B(\tfH)}
$$
show that  $\pi : \F \rightarrow \B(\tfH)$ is  unital. 
Since $v^* t_1 = t_1^* t_1 =  1_{\B(\fH)}$ and hence $ t_1^* v = 1_{\B(\fH)}$, 
we have 
\begin{gather*}
\tilde{s}_1^* \tilde{s}_1 
=  v^* t_1^* v v^* t_1 v
=  v^* 1_{\B(\fH)}\cdot  v 
=  v^* v = 1_{\B(\tfH)}, \quad \text{ and } \\ 
\tilde{s}_j^* \tilde{s}_j 
=  v^* t_j^* t_1^* v v^* t_1 t_j v
=  v^* t_j^* 1_{\B(\fH)}  t_j v 
=  v^* \pi(1_\T) v
= 1_{\B(\tfH)} \quad \text{ for } j\ge 2.
\end{gather*}
We also have
\begin{gather*}
\tilde{s}_1 \tilde{s}_1^* 
=  v^* t_1 v v^* t_1^* v
=  v^* t_1 \pi(1_\T) t_1^* v, \quad \text{ and } \\
\tilde{s}_j \tilde{s}_j^* 
=  v^* t_1 t_j \pi(1_\T) t_j^* t_1^* v 
=  v^* t_1 t_j t_j^*  t_1^* v
\quad \text{ for } j\ge 2
\end{gather*}
so that 
\begin{equation*}
\tilde{s}_1 \tilde{s}_1^* 
+ \sum_{j=2}^m \tilde{s}_j \tilde{s}_j^* 
=  v^* t_1 \left( \pi(1_\T) + \sum_{j=2}^m t_j t_j^* \right)  t_1^* v 
<  v^* t_1 1_{\B(\fH)} t_1^* v 
= v^* \pi(e) v = \tilde{\pi}(f).
\end{equation*}
(ii) 
Since the representation $\tilde{\pi}: \F \rightarrow \B(\tfH)$
and the isometries $\{\tilde{s}_j \}_{j \in \N}$ satisfy 
\eqref{eq:sjB},
Proposition \ref{prop:universalityforOHF} tells us that
there exists an injective $*$-homomorphism
$\Phi: \OHF\rightarrow \B(\tfH)$ satisfying  
$
\Phi(S_i) = \tilde{s}_i, \, i \in \N
$
and 
$
 \Phi(a) = \tilde{\pi}(a), \, a \in \F,
$
which gives rise to an isomorphism 
from 
$
\OHF 
$ to
$C^*(\tilde{\pi}(\F), \{\tilde{s}_j\}_{j\in \N}).
$

We will next prove that 
$\Phi(f \OHF f) = C^*(\pi(\T), \{ t_j \}_{j \in \N})$.
Put
$\tilde{f}: = \tilde{\pi}(f).$
By a completely similar manner to the proofs of Lemma \ref{UniversalhatA2}
and  Lemma \ref{UniversalhatA3},
one knows that 
\begin{equation}\label{eq:tildeftildepi}
\tilde{f}\, C^*(\tilde{\pi}(\F), \{\tilde{s}_j\}_{j\in \N})\, \tilde{f}
= C^*(\tilde{s}_1 \tilde{\pi}(\F) \tilde{s}^*_1, 
\tilde{s}_1 \tilde{f},  \{\tilde{s}_1 \tilde{s}_j\tilde{s}_1^* \}_{j \in \N}).
\end{equation}
So by using a similar manner to Lemma \ref{UniversalhatA4}, 
it suffices to prove that $\tilde{s}_1  \tilde{\pi}(\F) \tilde{s}_1^* = \pi(\T)$,
 $\tilde{s}_1 \tilde{f} = t_1$ and
$\tilde{s}_j \tilde{s}_1^* =t_j$ for $j \ge 2.$
Since 
$
\tilde{s}_1 v^* = v^* t_1 vv^* = t_1^* t_1 \pi(1_\T) = \pi(1_\T),
$
we have 
$$
\tilde{s}_1  \tilde{\pi}(\F) \tilde{s}_1^* 
=\tilde{s}_1 v^* \pi(\T) v  \tilde{s}_1^* = \pi(\T).
$$ 
We also  have
\begin{gather*}
\tilde{s}_1 \tilde{f} 
=  v^* t_1 \pi(e) v
=  t_1^* t_1\cdot t_1 t_1^*  v
=  1_{\B(\fH)}\cdot t_1 t_1^*  t_1
=t_1, \quad \text{ and } \\
\tilde{s}_j \tilde{s}_1^*
=  v^* t_1 t_j v v^* t_1^* v
=  1_{\B(\fH)} t_j \pi(1_\T) 1_{\B(\fH)}
=t_j \quad \text{ for } j\ge 2.
\end{gather*}  
By \eqref{eq:tildeftildepi}, we have
\begin{equation}\label{eq:tildeftildepi}
\tilde{f}\, C^*(\tilde{\pi}(\F), \{\tilde{s}_j\}_{j\in \N})\, \tilde{f}
= C^*(\pi(\T),\{ t_j \}_{j \in \N}).
\end{equation}
As
$\Phi(f \OHF f ) = \tilde{f}\, C^*(\tilde{\pi}(\F), \{\tilde{s}_j\}_{j\in \N})\, \tilde{f}$,
we have an isomorphism
 $f \OHF f \cong C^*(\pi(\T),\{ t_j \}_{j \in \N}).$
\end{proof}
Recall that the unital $C^*$-algebra $\T$ is defined by $S_1 \F S_1^*$  as a subalgebra of $\OHF$
and a projection $e \in \T$ is defined by $e= S_1 f S_1^*\in \calK(H)$. 
\begin{corollary}\label{cor:maincoro}
Suppose that there exist a representation
$\pi: \T \rightarrow \B(\fH)$ 
and a family $t_j \in \B(\fH), j \in \N$
of partial isometries on $\fH$ satisfying
\begin{equation*}
\begin{cases}
& (t1) \quad t_1^* t_1 = 1_{\B(\fH)}, \qquad
      t_j ^* t_j = \pi(1_\T) \quad  \text{ for } j\ge 2,
 \\
& (t2) \quad t_1 t_1^* = \pi(e), \qquad
\pi(1_\T) + \sum_{j=2}^m t_j t_j^* < 1_{\B(\fH)} \, \text{ for } 2 \le m \in \N.
\end{cases}
\end{equation*}
Then there exists an isomorphism
$\Phi: f \OHF f \rightarrow C^*(\pi(\T), \{ t_j \}_{j \in \N})$.
\end{corollary}
 Now we are in position to give a proof of Theorem \ref{thm:mainuniversalA3}.

 \begin{proof}[Proof of Theorem \ref{thm:mainuniversalA3}] 
 (i)
Let $\A$ be a unital Kirchberg algebra with finitely generated $\K$-groups.
One may take a separable unital nuclear UCT $C^*$-algebra $\F$ satisfying \eqref{eq:KFExtsA}
as in Section \ref{sect:Reviewreciprocal}.
 By \eqref{eq:constructionAhat} and Lemma \ref{UniversalhatA4}, 
the reciprocal dual $\whatA$ is generated by the unital subalgebra $\T$ and 
 partial isometries $\{ T_j \}_{j \in \N}$ 
 satisfying the relations in Lemma \ref{lem:relationsTjeV} (i).
As $\T = S_1 \F S_1$ and $e = S_1 f S_1^*$ with $S_1^* S_1 = 1_\F$, 
we have
\begin{equation}\label{eq:KFKT}
(\K_0(\F), [f]_0, \K_1(\F)) \cong
(\K_0(\T), [e]_0, \K_1(\T)),
\end{equation} 
showing \eqref{eq:1.1}. 
 In the relations (T2) of Lemma \ref{lem:relationsTjeV} (i),
 $T_1^* T_1 = f$ because $T_1 = S_1 f$,
 which is the unit of the algebra $f \OHF f$ and written as $1_{\whatA}$.
 Hence the relations (T1), (T2) in Lemma \ref{lem:relationsTjeV} (i) are the ones 
 appeared in \eqref{eq:main}. 
 This shows that the $C^*$-algebra $\whatA$ is the $C^*$-algebra satisfying the conditions 
 (a), (b) of Theorem \ref{thm:mainuniversalA3} (i).
 Its universality and uniqueness with respect to the subalgebra $\T$ and relations \eqref{eq:main} 
 follow from Corollary \ref{cor:maincoro}.
 
 (ii) Since 
 \begin{equation*}
 (\Exts^1(\A), [\iota_\A(1)]_s, \Exts^0(\A)) \cong
(\K_0(\whatA), [1_{\whatA}]_0, \K_1(\whatA)),
\end{equation*}
together with \eqref{eq:KFKT},
 we have
\begin{equation*}
(\K_0(\T), [e]_0, \K_1(\T)) \cong
(\K_0(\whatA), [1_{\whatA}]_0, \K_1(\whatA)).
\end{equation*} 
 because of \eqref{eq:KFExtsA}.
As the $C^*$-algebra $\whatA$ is a Kirchberg algebra, 
 its isomorphism class depends only on its 
 $K$-theory date with the position of the class of the unit in $\K_0$-group
 (\cite{Kirchberg}, \cite{Phillips}),
 the isomorphism class of $\whatA$ does not depend on the choice of the unital $C^*$-algebra $\T$
 as long as it satisfies \eqref{eq:1.1}. 
  \end{proof}  
  
\begin{remark}
If a Kirchberg algebra $\A$ is a simple Cuntz--Krieger algebra $\OA$,
one may take the unital $C^*$-algebra $\T$ as the Toeplitz algebra
$\T_{A^t}$ for the transposed matrix $A^t$. 
Then the universality of the reciprocal dual $\whatOA$ in Theorem \ref{thm:mainuniversalA3}
is nothing but the universality stating as in \cite[Theorem 1.2]{MatSogabe3}.
The detail will be discussed in Section \ref{sec:Examples}.
\end{remark}

\section{Ergodic automorphisms on $\whatA$ in terms of the generators}\label{sec:ergodicaut}
In \cite{MatSogabe4}, it was proved that there exists 
an aperiodic ergodic automorphism of an arbitrary unital Kirchberg algebra.
In this section, we will concretely realize the ergodic automorphism on the Kirchberg algebra with finitely generated $\K$-groups
in terms of their generators 
in the universal representation stated in Theorem \ref{thm:mainuniversalA3}.
It induces a concrete description of an ergodic automorphism of the Cuntz algebra $\O_2$, which will 
be stated in Section \ref{sec:Examples}.

As in Theorem \ref{thm:universalityforOHF}, 
let us represent $\OHF$ as the universal $C^*$-algebra 
$C^*_{\text{univ}}(\F, \{ S_n \}_{n \in \Z})$ with relations:
\begin{equation*}
S_n^* S_n = 1  \text{ for } n \in \Z \quad 
\text{ and } \quad
\sum_{|n| < m} S_n S_n^* < f \quad \text{ for any } m \in \N
\end{equation*}
where $f \in \calK(H) \subset \F$ is a fixed minimal projetion.
We note that in the above description the index set of the isometries $\{S_n \}_{n \in \Z}$ 
is the integer group $\Z$ instead of $\N$ and $S_0$ plays a role of $S_1$ in the previous section.
By Theorem \ref{thm:mainuniversalA3},  the reciprocal dual $\whatA$ of $\A$ 
is given by the corner $f \OHF f$ of $\OHF$.
Put
\begin{equation*}
\T: = S_0 \F S_0, \quad 
t_0 = S_0 f, \quad
e:= S_0 f S_0^*, \quad
\quad \text{ and }
\quad
t_n := 
S_n S_0^*  \text{ for } n \in \Z \text{ with } n \ne 0.
\end{equation*}
We have shown that 
$\whatA$ is realized as the universal $C^*$-algebra 
$C^*_{\text{univ}}(\T, \{t_n\}_{n \in \Z})$ 
satisfying  relations \eqref{eq:main}
such that 
 \begin{equation}\label{eq:univZ}
 \begin{split}
(1)\quad &  t_0^* t_0 =  1_{\widehat{\A}}, \qquad
            t_j ^* t_j = 1_\T \quad \text{ for } j \in \Z \text{ with } j \ne 0 \\
(2)\quad & t_0 t_0^* =e, \qquad  
1_\T  + \sum_{0<|j| \le m} t_j t_j^* < 1_{\widehat{\A}} \quad \text{ for } m \in \N, 
\end{split} 
\end{equation}
Define an automorphism $\theta$ 
on the $C^*$-algebra 
$C^*_{\text{univ}}(\F, \{ S_n \}_{n \in \Z})$ 
by
\begin{equation}\label{eq:autthetaOHF}
\begin{cases}
\theta(a) = a & \text{ for } a \in \F,\\
\theta(S_n) = S_{n+1} & \text{ for } n \in \Z.
\end{cases}
\end{equation}
By the universality of the $C^*$-algebra
$\OHF$, 
$\theta$ gives rise to an automorphism on $\OHF$.
We know the following lemma proved in
\cite[Lemma 3.2 and Lemma 3.3]{MatSogabe4}.
\begin{lemma}[{\cite[Lemma 3.2 and Lemma 3.3]{MatSogabe4}}] \label{lem:5.1}
The automorphism
$\theta^n$ on $\OHF$ for each $n \in \Z$ with $n\ne 0$
is outer, and the fixed point algebra $(\OHF)^\theta$ 
of $\OHF$ under $\theta$ is $\F$.
\end{lemma}
Since $\theta$ satisfies $\theta(f) = f$,
it induces an automorphism written $\hat{\theta}$
on $\whatA$ by restricting
$\theta$ to the subalgbera $f \OHF f$.
The automorphism $\hat{\theta}$ on $\whatA$ is nothing but the one constructed in 
\cite[Theorem 3.1]{MatSogabe4}.
Lemma \ref{lem:5.1} together with the minimality of the projection $f$ in $\F$
directly tells us that the fixed point algebra $(f \OHF f)^{\theta}$ of
$f \OHF f$ under $f$ is the scalar multiple of the unit $f$ (i.e., $f\mathcal{F}f=\mathbb{C}f$). 
The following result was proved in \cite{MatSogabe4} in a more general setting.
\begin{proposition}[{\cite[Theorem 3.1]{MatSogabe4}}]
The automorphism
$\hat{\theta}^n$ on $\whatA$ for each $n \in \Z$ with $n\ne 0$
is outer, and the fixed point algebra $(\whatA)^\theta$ 
of $\whatA$ under $\hat{\theta}$ is $\mathbb{C}1_{\whatA}$ (i.e., the automorphism $\hat{\theta}$ is aperiodic and ergodic on $\whatA$).
\end{proposition}
We will write $\hat{\theta}$ in terms of $\T $ and $ t_n , n\in \Z$
in the following way.
For $b \in \T = S_0 \F S_0^*$, so that 
$ b = S_0 a S_0^*$ for some $ a ( = S_0^* b S_0) \in \F$.
We then have  
for $n \in \Z$ with $n \ne 0$,
\begin{align*}
\hat{\theta}^n(b) = \theta^n(S_0 a S_0^*)
= S_n a  S_n^*
= S_nS_0^* S_0  a  S_0^* S_0 S_n^*
= t_n b t_n^*,
\end{align*}
so that $\hat{\theta}^n(b) = t_n b t_n^*$ for $b \in \T$ and $ n \in \Z, n \ne 0$.
In particular $\hat{\theta}(b) = t_1 b t_1^*$ for $b \in \T$.
We note that 
$ b = \hat{\theta}^0(b) \ne t_0 b t_0^* = S_0 b S_0^*$.
And also we have
\begin{align*}
\hat{\theta}(t_0 ) 
=  \theta(S_0)\theta(f)= S_1 S_0^* S_0 f = t_1 t_0, \qquad
\hat{\theta}(t_{-1} ) 
=  \theta(S_{-1})\theta(S_0^*)= S_0 S_1^* = t_1^*.
\end{align*} 
 Similarly
we have for $n \in \Z$ with $n \ne 0, -1$,
\begin{align*}
\hat{\theta}(t_n ) = \theta(S_n S_0^*)
= S_{n+1} S_1^*
= S_{n+1} S_0^* S_0 S_1^*
= t_{n+1} t_1^*.
\end{align*}
Therefore we have
\begin{proposition}\label{prop:ergodicauto}
In the reciprocal dual 
$\whatA = C^*_{\text{univ}}(\T, \{ t_n \}_{n \in \Z})$,
the correspondence $\hat{\theta}:\whatA \rightarrow \whatA$ defined by  
\begin{equation*}
\begin{cases}
\hat{\theta}(b)  & = t_1 b t_1^* \quad \text{ for } b \in \T, \\
\hat{\theta}(t_0)  & = t_1 t_0, \\
\hat{\theta}(t_{-1})  & = t_1^*, \\
\hat{\theta}(t_n)  & = t_{n+1} t_1^* \quad \text{ for } n \in \Z, n \ne 0, -1 
\end{cases}
\end{equation*}
gives rise to an aperiodic ergodic automorphism on $\whatA$.
\end{proposition}
\begin{remark}
It is easy to see that the conjugacy class of the automorphism $\hat{\theta}$ of $\whatA$
does not depend on the choice of a bijection between $\N$ and $\Z$ in exchanging the index set
of the family $\{t_j\}_j$ of partial isometries.
\end{remark}

\section{Invariant states}
In this section, we will show that a $\hat{\theta}$-invariant state on $\whatA$ is unique,
so that the automorphism $\hat{\theta}$ on $\whatA$ is uniquely ergodic.
In \cite[Theorem 1.5]{MatSogabe4},
it is shown that there exists a $\hat{\theta}$-invariant state on $\whatA$ in a more general setting.
The construction of the $\hat{\theta}$-invariant state in our setting is the following.
Let $\F$ be a separable unital nuclear UCT $C^*$-algebra containing $\calK(H)$ as an essential ideal
and  $f \in \calK(H)$ a minimal projection  satisfying \eqref{eq:KFExtsA}. 
Then we have a $C^*$-algebra $\T_{H_\F}$ as in the proof of Lemma \ref{lem:UniversalAhat1}.
Let $E_\F: \T_{H_\F}\rightarrow \F$ 
be the conditional expectation arising from a projection on the Fock space
$\E_\F =\F \oplus \bigoplus_{n=1}^\infty H_\F^{\otimes n}$ 
onto the subspace $\F$.
As $f$ is a minimal projection in $\F$,
there exists a scalar $\varphi_{\whatA}(f x f) \in \mathbb{C}$
for $x \in \O_{H_\F}$ satisfying 
\begin{equation}\label{eq:invariantstate}  
f E_\F(f x f) f = \varphi_{\whatA}(f x f) f \quad \text{ for } \, \, x \in \T_{H_\F}.
\end{equation}
Under the identification between $\T_{H_\F}$ and $\O_{H_\F}$,  
we have a state $\varphi_{\whatA}$ on $f \O_{H_\F} f = \whatA$,
which was proved  in \cite{MatSogabe4} to be invariant under the automorphism $\hat{\theta}$.

Let us represent the reciprocal algebra $\whatA$ as $C^*_{\text{univ}}(\T, \{ t_n \}_{n \in \Z})$
as in the previous section.
We will show the following proposition.
\begin{proposition}\label{prop:nvariantstate2}
Let $\varphi$ be a state on $\whatA$ such that
$\varphi\circ \hat{\theta} =\varphi$.
Let $a \in \whatA$ be a finite word of elements of 
$\T \cup \{ t_j\}_{j \in \Z} \cup \{ t_j^*\}_{j \in \Z}$ 
such that  $\varphi(a) \ne 0$.
 Then we have $ a \in \mathbb{C} 1_{\whatA}$.
\end{proposition}
To prove the proposition, we provide two lemmas.
The next one is straightforward from Proposition \ref{prop:ergodicauto}.
\begin{lemma}\label{lem:Invariantstate1}
\begin{equation*}
\hat{\theta}(t_n t_n^*)
=
\begin{cases}
t_1 e t_1^*  & \quad \text{ for } n=0, \\
1_\T & \quad \text{ for } n= -1, \\
t_{n+1} t_{n+1}^*   & \quad \text{ for } n\ne 0, -1. 
\end{cases}
\end{equation*}
\end{lemma} 
By using the lemma above, we have the following lemma.
\begin{lemma}\label{lem:invariantstate2}
Let $\varphi$ be a state on $\whatA$ such that
$\varphi\circ \hat{\theta} =\varphi$.
Then we have 
$\varphi(t_j t_j^*) =0$ for all $j \in \Z,$
and
$\varphi(b) =\varphi(t_i^* t_i) =0$ 
for all $b \in \T$ and $i \in \Z$ with $i \ne 0$.
\end{lemma}
\begin{proof}
We first see that $\varphi(t_1 t_1^*) =0.$
Put $\epsilon =\varphi(t_1 t_1^*).$
By Lemma \ref{lem:Invariantstate1}, 
we have for $n \in \N$
\begin{equation}\label{eq:varphin}
\varphi(t_{n+1} t_{n+1}^*) 
=\varphi\circ \hat{\theta}^n(t_1 t_1^*) 
=\varphi(t_1 t_1^*)
=\epsilon.
\end{equation}
The inequality $1_\T + \sum_{n=1}^m t_j t_j^* < 1_{\whatA}$ 
tells us that 
$\varphi(1_\T) + m \epsilon <1$ for all $m \in \N$, so that 
we have $\epsilon =0$, 
and \eqref{eq:varphin} shows $\varphi(t_{n} t_{n}^*) =0$ for all $n \in \N$.
Since
\begin{equation*}
0 \le \varphi(t_0 t_0^*) 
=\varphi\circ \hat{\theta}(t_0 t_0^*) 
=\varphi(t_1e t_1^*)
\le
\varphi(t_1 t_1^*), 
\end{equation*}
we have
$\varphi(t_0 t_0^*)  =0.$
We also have
\begin{equation}\label{eq:6.2}
 0 \le
 \varphi(t_{-1} t_{-1}^*) 
=\varphi\circ \hat{\theta}^2(t_{-1} t_{-1}^*) 
=\varphi\circ \hat{\theta}(1_\T)
=\varphi(t_1 1_\T t_1^*)
\le
\varphi(t_1 t_1^*)
\end{equation}
and smilarly for $n \ge 2$ we have
\begin{equation*}
 \varphi(t_{-n} t_{-n}^*) 
=\varphi\circ \hat{\theta}^{n-1}(t_{-n} t_{-n}^*) 
=\varphi(t_{-1} t_{-1}^*)
=0.
\end{equation*}
Therefore
$\varphi(t_j t_j^*) =0$ for all $j \in \Z.$
As in \eqref{eq:6.2}, we have
$\varphi(1_\T) = \varphi\circ\hat{\theta}(1_\T) =0$,
so that 
$$
|\varphi(b)| = |\varphi( 1_\T b )| \le \varphi(1_\T)^{\frac{1}{2}} \varphi(b^* b)^{\frac{1}{2}} =0
\quad \text{ for } b \in \T.
$$
\end{proof}
\begin{proof}[Proof of Proposition \ref{prop:nvariantstate2}]
Recall that $e\in\T$ is a minimal projection (i.e., $e\T e=\mathbb{C}e$).
Let 
$a\in C^*_{\text{univ}}(\T, \{t_j\}_{j \in \Z})$ 
be  of the form
\begin{equation}\label{eq:forma1}
a = a_1 a_2 \cdots a_m \quad \text{ where } \quad a_i \in \T \cup \{ t_j \}_{j \in \Z} \cup \{ t_j^* \}_{j \in \Z}.
\end{equation}
Assume that $\varphi(a) \ne 0$.
If $a_1 \in \T\cup \{ t_j \}_{j \in \Z}\cup \{ t_j^* \}_{j \in \Z, j\ne 0}$,
the inequality
\begin{equation*}
|\varphi(a) | = |\varphi(a_1 a_2 \cdots a_m ) | 
\le \varphi(a_1 a_1^*)^{\frac{1}{2}} \varphi((a_2 \cdots a_m)^* (a_2 \cdots a_m) )^{\frac{1}{2}} 
\end{equation*}
together with Lemma \ref{lem:invariantstate2}
shows us 
$\varphi(a_1 a_1^*) =0$ so that $\varphi(a)=0$, 
a contrdiction to the hypothesis that 
$\varphi(a) \ne 0.$
If $a_m \in \T\cup \{ t_j^* \}_{j \in \Z}\cup \{ t_j \}_{j \in \Z, j \ne 0}$,
the inequality
\begin{equation*}
|\varphi(a) | = |\varphi(a_1 \cdots a_{m-1} a_m ) | 
\le  \varphi((a_1 \cdots a_{m-1}) (a_1 \cdots a_{m-1})^*) ^{\frac{1}{2}} 
\varphi(a_m^* a_m)^{\frac{1}{2}}
\end{equation*}
similarly shows us that 
$\varphi(a_m^* a_m)=0$ and hence $\varphi(a) =0$.
Therefore we have $a_m =t_0$ as well as $a_1 = t_0^*$,
 so that the element $a$ must be of the form 
 \begin{equation}\label{eq:forma2}
a = t_0^* a_2 \cdots a_{m-1}t_0 \quad \text{ where } \quad a_i \in \T \cup \{ t_j \}_{j \in \Z} \cup \{ t_j^* \}_{j \in \Z}.
\end{equation}
We have several cases:

Case A : $a_2 \in \T.$

We have two subcases.

Case A-1 : $a_3, \cdots, a_{m-1} \in \T.$

Since $t_0 t_0^* =e$, we have 
$a = t_0^* a_2 \cdots a_{m-1}t_0
= t_0^* e  a_2 \cdots a_{m-1} e t_0$.
Now the condition $a_2 \cdots a_{m-1} \in \T$ implies that 
$e  a_2 \cdots a_{m-1} e  = c e$ for some $c \in \mathbb{C}$
because $e$ is a minimal projection in $\T$.
Hence 
$a = c t_0^*  e t_0 = c 1_{\whatA} \in \mathbb{C} 1_\whatA.$

Case A-2: $a_2,\dots, a_{k-1}\in \T$ and $a_k \not\in \T$ for some $k \ge 3$.

Put $b: = a_2\cdots a_{k-1} \in \T$ so that 
$ a = t_0^* b a_k a_{k+1} \cdots a_{m-1} t_0$. 
As $ b a_k \ne 0,$
we have
$a_k=t_0$ or $t_j^*$ for some $j \in \Z$.
In case $a_k = t_0,$ we have $t_0^* b t_0 = t_0^* e b e t_0$ with $b \in \T$.
As $e b e = c e$ for some $c \in \mathbb{C}$, we have 
  $t_0^* b a_k = c 1_{\whatA}$ so that 
$$
a = c a_{k+1} \cdots a_{m-1} t_0\quad \text{ for some } c \in \mathbb{C}, \, k\ge 3.
$$  
By the preceding argument as in the precedure to reduce \eqref{eq:forma2},
we have
$a_{k+1} = t_0^*$ so that $a$ is of the form
 \begin{equation}\label{eq:formaA1}
a = c t_0^* a_{k+2} \cdots a_{m-1}t_0 \quad \text{ for some } c \in \mathbb{C}, \, k\ge 3.
\end{equation}
In case $a_k = t_j^*,$ we have 
\begin{equation}\label{eq:formaA2}
a = t_0^* b t_j^* a_{k+1} \cdots a_{m-1}t_0 \quad \text{ for some } b \in \T, \, k\ge 3.
\end{equation}

Case B: $a_2 \in \{t_j\}_{j \in \Z}.$

Let $a_2 = t_j$ for some $j \in \Z$.
We then have $ a = t_0^* t_j a_3\cdots a_{m-1} t_0$,
and hence 
$j=0$ because $a\ne 0$.
Hence 
$a = 1_{\whatA} a_3 \cdots a_{m-1} t_0 =  a_3 \cdots a_{m-1} t_0$.
As in \eqref{eq:forma2}, we have 
$a_3 = t_0^*$ so that   
 \begin{equation}\label{eq:forma3}
a = t_0^* a_4 \cdots a_{m-1}t_0. 
\end{equation}

Case C: $a_2 \in \{t_j^*\}_{j \in \Z}.$

Let $a_2 = t_j^*$ for some $j \in \Z$.
We then have $ a = t_0^* t_j^* a_3\cdots a_{m-1} t_0$.

We have three subcases:

Case C-1: 
$a_3 \in \T.$

Since $t_j^* a_3 =0$ unless $j =0$, so that $j=0$ and hence 
 \begin{equation}\label{eq:forma4}
a = t_0^* t_0^* a_3 \cdots a_{m-1}t_0.
\end{equation}

Case C-2:  
$a_3 \in \{t_i\}_{i \in \Z}.$

Since $t_j^* t_i \ne 0$ if and only if $i=j$, and 
$t_j^* t_j = 1_{\T}$ for $ j \ne 0$ and $t_0^* t_0 = 1_{\whatA}$, we have
$a_1 a_2 a_3 = t_0^* t_j^* t_j = t_0^*$, so that 
 \begin{equation}\label{eq:forma5}
a = t_0^* a_4 \cdots a_{m-1}t_0.
\end{equation}

Case C-3: 
$a_3 \in \{t_i^*\}_{i \in \Z}.$

We then have
 \begin{equation}\label{eq:forma6}
a = t_0^* t_j^* t_i^* a_4 \cdots a_{m-1}t_0.
\end{equation} 
Consequently 
by using induction on the number $m$ with $a = a_1 a_2 \cdots a_m$
and by 
\eqref{eq:formaA1}, \eqref{eq:formaA2} in Case A-2,
\eqref{eq:forma3} in Case B,
and
\eqref{eq:forma4}, \eqref{eq:forma5}, \eqref{eq:forma6} in Case C,
the element $a$ may be written as in the following two forms
\begin{equation}
a  = t_0^* t_{k_1}^* \cdots t_{k_n}^* b_0 t_{j_l} \cdots t_{j_1} t_0, \label{eq:formaD1}
\end{equation}
or
\begin{equation}
a  = t_0^* b_1 t_{k_1}^* b_2 t_{k_2}^*  \cdots b_n t_{k_n}^*  d_l t_{j_l} d_{l-1} \cdots d_1 t_{j_1} d_0 t_0, 
\label{eq:formaD2}
\end{equation}
for some $b_0, b_1, \dots, b_n, \, d_0, d_1, \dots, d_l \in \T$.
Since 
$t_k^* b_i =0$ for $k\ne 0$,
and
$d_i t_j =0$ for $j\ne 0$,
the condition $t_{k_n}^*  b t_{j_l} \ne 0$ for sone $b \in \T$ 
forces us to 
$t_{k_n}^*  b t_{j_l} = t_0^*  b t_0 = t_0^* e b e t_0 = c e$ for some $ c \in \mathbb{C}.$ 
Hence \eqref{eq:formaD2} goes to 
\begin{equation*}
a  = c t_0^* b_1 t_{k_1}^* b_2 t_{k_2}^*  \cdots b_n  e d_{l-1} \cdots d_1 t_{j_1} d_0 t_0. 
\label{eq:formaD21}
\end{equation*}
One inductively has that the element of the forms  \eqref{eq:formaD2},
 and similarly \eqref{eq:formaD1}, 
is reduced to the following form
\begin{equation}\label{eq:formaD3}
a = c \overbrace{t_0^*\cdots t_0^*}^{p}\overbrace{t_0\cdots t_0}^{q} \quad 
\text{ for some } c \in \mathbb{C} \text{ and } p, q \in \N  
\end{equation}
because $t_0^* e t_0 = 1_{\whatA}$ and $t_0^* b t_0 \in \mathbb{C}1_{\whatA}$ for $b \in \T$.
As 
$| \varphi(t_0^n)| \le \varphi(t_0 t_0^*)^{\frac{1}{2}}\varphi(t_0^{n-1} t_0^{n-1*})^{\frac{1}{2}},
$
we have
$\varphi(t_0^{n}) = 0$ and similarly 
$ \varphi(t_0^{*n}) = 0$ for $n \in \N$
by Lemma \ref{lem:invariantstate2}.
We thus conclude that $p=q$ and $a \in \mathbb{C}1_{\whatA}$ from \eqref{eq:formaD3}.
\end{proof} 

\begin{proof}[Proof of Theorem \ref{thm:maininvariantstate}]
 Since the set of linear combinations of 
 finite words of elements of $\T \cup \{ t_j\}_{j \in \Z} \cup \{ t_j^* \}_{j \in \Z}$ 
 is dense in $\whatA$,
Proposition \ref{prop:nvariantstate2} shows that 
a $\hat{\theta}$-invariant state of $\widehat{\A}$ is unique and it has to be $\varphi_{\widehat{\A}}$.

 We will next show that the state $\varphi_{\whatA}$ on $\whatA$ is pure.
 Let $\F$ be a separable unital $C^*$-algebra containing $\calK(H)$ as an essential ideal
and  $f \in \calK(H)$ a minimal projection  satisfying \eqref{eq:KFExtsA}. 
Take a unit vector $\xi_0 \in H$ such that $ f = \xi_0 \otimes \xi_0^*$.
 Let $H_\F$ be the Hilbert $C^*$-bimodule 
 $H \otimes _{\mathbb{C}}\ell^2(\Z) \otimes _{\mathbb{C}} \F$ over $\F$.
  One may regard $f$ as an element of 
 $\O_{H_\F}= \T_{H_\F}$ as in \eqref{eq:eeA}.  
 Let $L^2(\whatA, \varphi_{\whatA})$ be the Hilbert space of the GNS-representation of $\whatA$ by the state 
 $\varphi_{\whatA}$.
 Under the identification 
 $\whatA = f \O_{H_F} f$ as in  
\eqref{eq:constructionAhat}, one may naturally identify 
 $L^2(\whatA, \varphi_{\whatA})$ with the Hilbert space
 spanned by 
\begin{align*}
fF(H_\F)f 
:= & f\F f\Omega \oplus (\xi_0 \otimes_{\mathbb{C}} \ell^2(\Z) \otimes_{\mathbb{C}} \F f) \\
   & \oplus 
 \bigoplus_{k=2}^\infty 
 (\xi_0 \otimes_{\mathbb{C}}\ell^2(\Z) \otimes_{\mathbb{C}} \F) \otimes
(H \otimes_{\mathbb{C}}\ell^2(\Z) \otimes_{\mathbb{C}} \F)^{\otimes k-2}
\otimes (H \otimes_{\mathbb{C}}\ell^2(\Z) \otimes_{\mathbb{C}} \F f)  
\end{align*}
 through the correspondence 
 $$
 f T_\xi f \in f \T_{H_\F} f = f \O_{H_\F} f \longrightarrow f \xi f \in f F(H_\F) f 
 \quad \text{ for } \xi \in H_\F
 $$
where $\Omega$ is a vacuum vector.
Note that $fF(H_\F)f$ is a right Hilbert $f\F f$-module and $f\F f=\mathbb{C}f$.
 For any $X \in \B(L^2(\whatA, \varphi_{\whatA}))\cap ({\widehat{\A}})^\prime$ and
 $\xi \in (\xi_0 \otimes _{\mathbb{C}}\ell^2(\Z) \otimes _{\mathbb{C}} \F)\otimes (H \otimes _{\mathbb{C}}\ell^2(\Z) \otimes _{\mathbb{C}} \F)^{\otimes k-2}
\otimes (H \otimes _{\mathbb{C}}\ell^2(\Z) \otimes _{\mathbb{C}} \F f)
$ with $k\ge 2$, we have
\begin{equation*}
< X \Omega \mid \xi > 
= 
< X \Omega \mid T_\xi \Omega > 
= 
< X T_\xi^* \Omega \mid \Omega > 
= 0
\end{equation*}  
so that $ X \Omega = \lambda \Omega$ for some $\lambda \in \mathbb{C}.$
Since $X\xi = X T_\xi \Omega = T_\xi X\Omega = \lambda \xi$,
we have
$X = \lambda \in \mathbb{C}$, showing that 
$\B(L^2(\whatA, \varphi_{\whatA})\cap \whatA^\prime = \mathbb{C}.$
 Hence the state $\varphi_{\whatA}$ is pure.

 Since any unital Kirchberg algebra $\A$ with finitely generated K-groups is the reciprocal dual of 
 the reciprocal dual of itself (i.e., $\A \cong \widehat{\whatA}$),
we reach the conclusion of Theorem \ref{thm:maininvariantstate}.
\end{proof}

\begin{remark}
{\bf 1.} As in \cite{AbadieDykema}, an automorphism $\alpha$ of a unital $C^*$-algebra $\A$ 
is uniquely ergodic if and only if for every $a \in \A$
the ergodic sums
$\frac{1}{n}\sum_{i=0}^{n-1}\alpha^i(a)$ converge in norm to a scalar multiple of the unit $1_{\A}$
of $\A$ as $n$ goes to infinity.   
By Theorem \ref{thm:universalityforOHF}, 
one may represent $\OHF$ as the universal $C^*$-algebra 
$C^*_{\text{univ}}(\F, \{ S_n \}_{n \in \Z})$ with relations \eqref{eq:SiB}.
Consider the automorphism $\theta$ on $\OHF$ defined by \eqref{eq:autthetaOHF}.
The restriction of $\theta$ to $f \OHF f$ is our ergodic automorphism $\hat{\theta}$
on $\whatA$.
By using the fact that the family
\begin{equation}\label{eq:generatotsSnFremark}
S_n \F S_n^*, \qquad S_n f, \qquad S_n S_i S_n^*, \qquad n, i\in \N
\end{equation}
 generates $f \OHF f$ as in Lemma \ref{UniversalhatA2},
one may prove that 
the ergodic sums
$\frac{1}{n}\sum_{i=0}^{n-1}\hat{\theta}^i(a)$ 
converge in norm to a scalar multiple of the unit $1_{f\OHF f}$
on the  dense subalgebra of $\OHF$
algebraicaly generated by the elements in \eqref{eq:generatotsSnFremark}
by overlapping discussions  with our proof of Proposition  \ref{prop:nvariantstate2}.
Consequently, it is possible to prove that  
the ergodic sums
$\frac{1}{n}\sum_{i=0}^{n-1}\hat{\theta}^i(a)$ 
converge in norm to a scalar multiple of the unit $1_{f\OHF f}$ of $f \OHF f$
fo all elements of $f \OHF f$, showing that $\hat{\theta}$ is uniquely ergodic on $\whatA$.

{\bf 2.} Recall that there is another characterization of a state $\varphi$ on a $C^*$-algebra $\A$  
to be pure, that is, every positive linear functional majorized by $\varphi$ 
is a positive scalar multiple of $\varphi$ (cf. \cite[Definition 9.21]{Takesaki}).
By using the characterization of a pure state, 
one may show that $\varphi_{\whatA}$ is pure in the following way. 
 Let $\phi$ be a positive linear functional on $\whatA$ such that 
 $\phi \le \varphi_{\whatA}$.  
Since the stae $\varphi_{\whatA}$ satisfies the conclusion of Lemma \ref{lem:invariantstate2},
so does the positive linear functional $\phi$.   
Hence the proof of Proposition \ref{prop:nvariantstate2} works for $\phi$
so that  one may conclude that $\frac{\phi}{\phi(1)}$ coincides with $\varphi_{\whatA}$.
\end{remark}

\section{Examples}\label{sec:Examples}

{\bf 1.} The reciprocal dual $\whatOA$ of a simple Cuntz--Krieger algebra $\OA$:

Let $\A$ be a simple Cuntz--Krieger algebra $\OA$
for an irreducible non-permutation matrix $A =[A(i,j)]_{i,j=1}^N$ 
 with entries in 
$\{0,1 \}$.
For $m \in \N$, let us denote by 
$
B_m(X_A)
$
the set of words 
$(\mu_1,\dots,\mu_m), \, \mu_i =1,\dots, N
$
of length $m$ such that 
$A(\mu_i, \mu_{i+1}) =1, i=1,\dots, m-1.
$
We put
$B_*(X_A) = \cup_{m=0}^\infty B_m(X_A)$, 
where $B_0(X_A)$ denotes the empty word.
The Toeplitz algebra $\TA$ for the matrix $A$ 
introduced in \cite{EFW} (cf. \cite{Evans}) is the universal unital $C^*$-algebra 
generated by $N$-partial isometries $T_1,\dots, T_N$ and a nonzero projection $e$ 
subject to the relations:
\begin{equation}\label{eq:Toeplitz}
\sum_{j=1}^N T_j T_j^* + e = 1, \qquad
T_i^* T_i = \sum_{j=1}^N A(i,j) T_j T_j^* + e, \quad i=1,\dots,N.
\end{equation} 
For $\mu =(\mu_1,\dots,\mu_m)\in B_m(X_A),$ 
we write
$T_{\mu_1}\cdots T_{\mu_m}$ as $T_\mu$.
As the identities  
\begin{equation}\label{eq:Toeplitzidentity}
T_\mu e T_\nu^* \cdot T_\xi e T_\eta^*
=
\begin{cases}
T_\mu e T_\eta^* & \text{ if } \nu = \xi, \\
0 & \text{ otherwise}
\end{cases}
\end{equation}
for $\mu, \nu, \xi, \eta \in B_*(X_A)$ hold,  
$\TA$ is the closed linear span of elements of the form
$$ T_\mu e T_\nu^*, \quad T_\xi T_\eta^*,  \qquad \mu, \nu, \xi, \eta \in B_*(X_A)
$$
and 
the $C^*$-subalgebra of $\TA$ generated by 
$ T_\mu e T_\nu^*, \mu, \nu \in B_*(X_A)$
is isomorphic to the $C^*$-algebra $\calK(H)$
of compact operators on a separable infinite dimensional
Hilbert space $H$, such that there is a natural short exact sequence
\begin{equation*}
 0 \longrightarrow \calK(H) \longrightarrow \T_A \longrightarrow \OA \longrightarrow 0.
\end{equation*}
We will consider the Toeplitz algebra $\TAT$ for the transposed matrix 
$A^t =[A^t(i,j)]_{i,j=1}^N$ of $A$.  
It was proved in \cite{MaJMAA2024} and \cite{MaAnalMath2024} (cf. \cite{MatSogabe}, \cite{MatSogabe2}) that
\begin{equation}
(\K_0(\TAT), [e]_0, \K_1(\TAT)) \cong (\Exts^1(\OA), [\iota_{\OA}(1)]_s, \Exts^0(\OA)).
\end{equation}
Hence the reciprocal dual $\whatOA$ is realized as a universal $C^*$-algebra 
$C^*_{\text{univ}}(\T_{A^t}, \{t_j\}_{j \in \N})$ generated by $\TAT$ and a family 
 $\{t_j\}_{j \in \N}$ of partial isometries 
 satisfying relations \eqref{eq:main}.
 The algebra $\TAT$ itself is a universal $C^*$-algebra generated by 
 partial isometries $T_j, j=1,\dots N$ and a projection $e$ satisfying \eqref{eq:Toeplitz}
 for the transposed matrix $A^t$.
 By unifying the generators $T_j, j=1,\dots, N$ and $\{t_j\}_{j \in \N}$
 with their relations \eqref{eq:Toeplitz} and \eqref{eq:main}, 
 we have the following realization of the reciprocal dual
  $\whatOA$ as a simple Exel--Laca algebra \cite{EL}
 in the following way, which has been already seen in \cite[Theorem 1.3]{MatSogabe3}. 
\begin{proposition}[{\cite[Theorem 1.3]{MatSogabe3}}] \label{prop:7.1}
The  reciprocal dual $\whatOA$ is the simple Exel--Laca algebra 
$\O_{\widehat{A}}$ defined by the infinite matrix 
$\widehat{A}$ 
over $\N$ with entries in $\{0,1\}$ such that
\begin{equation}\label{eq:matrixAinfty}
\widehat{A}(i,j) :=
\begin{cases}
A(j,i) & \text{ if } 1 \le i , j \le N, \\
1 & \text{ if } 1 \le i \le N,\, \, j= N+1, \\ 
0 & \text{ if } 1 \le i \le N,\, \, j\ge  N+2, \\ 
1 & \text{ if }  i =N+1, \\ 
1 & \text{ if }  i \ge N+2,\, \, j \le  N+1, \\ 
0 & \text{ if }  i \ge N+2,\, \, j \ge N+2
\end{cases}
\end{equation}
for $i,j \in \N$.
The matrix $\widehat{A}$ is written as  
\begin{equation*}
\widehat{A}=
\left[
\begin{array}{ccccccc}
&                &1      &0     &0     &\cdots \\
&\text{\Huge A}^t&\vdots&\vdots&\vdots&\vdots \\
&                &1      &0     &0     &\cdots \\
1&\dots          &1      &1     &1     &\cdots \\
1&\dots          &1      &0     &0     &\cdots \\
1&\dots          &1      &0     &0     &\cdots \\
\vdots &\dots    &\vdots &\vdots&\vdots&\vdots \\
\end{array}
\right].
\end{equation*}
\end{proposition}
\begin{proof}
Theorem \ref{thm:mainuniversalA3} (i) tells us that the $C^*$-algebra $\whatOA$ 
is the universal $C^*$-algebra
$C^*_{\text{univ}}(\TAT, \{t_j\}_{j \in \N})$ generated by $\TAT$ 
and partial isometries  
$\{t_j\}_{j \in \N}$ satisfying relations \eqref{eq:main}
which are 
\begin{equation} \label{eq:TATtjrelations}
t_1^* t_1 = 1_{\whatOA}, \quad t_j^* t_j = 1_{\TAT} \text{ for } j \ge 2, \quad t_1 t_1^* = e, \quad
1_{\TAT} + \sum_{j=2}^m t_j t_j^* < 1_{\whatOA} 
\end{equation}
for $m \ge 2$.
Let us define a family of partial isometries $\{ s_j \}_{j \in \N}$ by setting
\begin{align*}
s_j =
\begin{cases}
T_j & \text{ for } 1\le j \le N, \\
t_{j-N} & \text{ for } j>N.
\end{cases}
\end{align*}
It is direct to see that 
the $C^*$-algebra $C^*(\{s_j \}_{j \in \N})$ generated  by the partial isometries $\{ s_j \}_{j\in \N}$
is the Exel--Laca algebra $\O_{\widehat{A}}$ defined by the matrix $\widehat{A}$.
As $C^*(\{s_j \}_{j \in \N})$ coincides with $C^*_{\text{univ}}(\TAT, \{t_j\}_{j \in \N})$,
we have $\whatOA$ is the Exel--Laca algebra  $\O_{\widehat{A}}$ for the matrix $\widehat{A}$.
\end{proof}
\begin{remark}
The matrix $\whatA$ defined in \eqref{eq:matrixAinfty}
was written as $\whatA_\infty$ in \cite[Theorem 1.3]{MatSogabe3}.
\end{remark}
The class of the reciprocal duals $\whatOA$ of simple Cuntz--Krieger algebras
is exactly the class of unital Kirchberg algebras $\A$ with finitely generated $\K$-groups
having torsion free $\K_1$-groups such that  
$ \rank(\K_0(\A)) - \rank(\K_1(\A)) =1$
(\cite[Proposition 4.4]{MatSogabe3}).
Hence Proposition \ref{prop:7.1} shows that the Kirchberg algebras belonging to the class 
are explicitly represented 
as Exel--Laca algebras by the above matrices.

\medskip

{\bf 2.} The reciprocal dual $\widehat{\calP}_\infty$ of $\PI$:

Let $\PI$ be the Kirchberg algebra such that 
$\K_0(\PI) = 0, \, \K_1(\PI) = \Z$. 
It is  a unique Kirchberg algebra 
 realized as the unital Exel--Laca algebra
$\O_{P_\infty}$ by the infinite matrix $P_\infty$
\begin{equation}\label{eq:Pinftymatrix}
P_\infty =
\begin{bmatrix}
1 & 0 & 1 & 1 & 1 & 1 &\cdots &   & \\
0 & 1 & 1 & 1 & 1 & 1 &\cdots &   & \\
1 & 0 & 1 & 0 & 0 & 0 &\cdots &   & \\
0 & 1 & 0 & 1 & 0 & 0 &\cdots &   & \\
0 & 0 & 1 & 0 & 1 & 0 &\cdots &   & \\
  &\ddots &\ddots &\ddots   &\ddots &\ddots&\ddots &   & \\            
  &   &   &   &   & & &   &             
\end{bmatrix}
\quad (\text{cf. }  \cite[\text{Example }4.2]{RaebSzy}).
\end{equation}
We will concretely represent the reciprocal dual $\widehat{\calP}_\infty$
of $\calP_\infty$ by generators and its relations 
as a unital simple Exel--Laca algebra.
We know that $\Extw^1(\PI) =\Z, \Extw^0(\PI) =  0 $ by the UCT (cf. \cite[Theprem 23.1.1]{Blackadar})
so that the cyclic six-term exact sequence \eqref{eq:cyclic6temExtsw} for $\A = \PI$
goes to
\begin{equation}\label{eq:6termExtPI2}  
\begin{CD}
\Exts^0(\PI) @>>> 0 @>>> \Z  \\
@AAA @. @VV{\iota_\PI}V \\
0 @<<< \Z @<<< \Exts^1(\PI) 
\end{CD}.
\end{equation}
Since 
there is an isomorphism between $\Exts^1(\PI)$ and $\Z\oplus \Z$ such that 
$\iota_{\PI}(1) $ is mapped to $0\oplus 1\in \Z\oplus \Z,$ 
 we have
\begin{equation*}
(\Exts^1(\PI), [\iota_{\PI}(1)]_s, \Exts^0(\PI))
\cong
(\Z\oplus \Z,0\oplus 1,  0 ). 
\end{equation*} 
Let
$\T_\infty$ be the  
universal unital $C^*$-algebra generated by 
a countable family $T_i, i \in \N$ 
of isometries and one nonzero projection $e$
such that
\begin{equation} \label{eq:relationTinfty}
\sum_{j=1}^m  T_j T_j^* + e <1 \quad \text{ for } m \in \N.
\end{equation}
Since 
\begin{equation*}
T_\mu e T_\nu^* T_\xi e T_\eta^*
= 
\begin{cases}
T_\mu e T_\eta^* & \text{ if } \nu = \xi, \\
0 & \text{ otherwise,}
\end{cases} 
\end{equation*}
the $C^*$-subalgebra 
$C^*(T_\mu e T_\nu^* \mid \mu, \nu \in \N^*)$
 generated by partial isometries 
$
T_\mu e T_\nu^*, \mu, \nu \in \N^*
$
is isomorphic to the $C^*$-algebra $\calK(H)$ of compact operators.
It forms an essential ideal of $\T_\infty$ such that the sequence 
\begin{equation}\label{eq:Tinfty}
0 \longrightarrow
\calK(H) \longrightarrow
\T_\infty \longrightarrow
\OI \longrightarrow
0
\end{equation} 
is exact. 
That is, the exact sequence \eqref{eq:Tinfty} 
is a trivial extension of $\OI$.
The $\K$-theory standard  six-term exact sequence for \eqref{eq:Tinfty}
together with 
$\K_0(\OI) = \Z, \K_1(\OI) =  0$
yields 
the cyclic six term exact sequence
\begin{equation}\label{eq:standard6termforPI2} 
\begin{CD}
\K_1(\T_\infty) @>>> 0 @>>> \Z \\
   @AAA @. @VVV\\
  0 @<<< \Z @<<< \K_0(\T_\infty).
\end{CD}
\end{equation}
As the downward arrow in the diagram \eqref{eq:standard6termforPI2}
is induced by the map
$ \Z \ni 1 \rightarrow [e]_0 \in \K_0(\TI)$,
there is an isomorphism between $\K_0(\TI)$ and $\Z\oplus \Z$ such that 
$[e]_0 $ is mapped to $0\oplus 1\in \Z\oplus \Z,$
so that 
\begin{equation*}
(\K_0(\TI), [e]_0, \K_1(\TI))
\cong
(\Z\oplus \Z,0\oplus 1, 0), 
\end{equation*} 
showing that 
\begin{equation*}
(\K_0(\TI), [e]_s, \K_1(\TI))
\cong
(\Exts^1(\PI), [\iota_{\PI}(1)]_s, \Exts^0(\PI)). 
\end{equation*} 
By Theorem \ref{thm:mainuniversalA3},
the reciprocal dual $\whatPI$ is realized as a universal $C^*$-algebra 
$C^*_{\text{univ}}(\TI, \{t_j\}_{j \in \N})$ generated by $\TI$ and a family 
 $\{t_j\}_{j \in \N}$ of partial isometries 
 satisfying relations \eqref{eq:main} 
 for $\T = \TI$.
 The trivial extension $\TI$ of $\OI$  itself is a universal $C^*$-algebra generated by 
 isometries $T_j, j=1,\dots N$ and a nonzero projection $e$  satisfying 
  \eqref{eq:relationTinfty}.
 By unifying them,
we have the following relations:
\begin{gather*}
T_i^* T_i = 1_{\TI}, \quad  i \in \N, \quad
\sum_{i=1}^m T_i T_i^* + e <1_{\TI}, \\
t_1^* t_1 = 1_{\whatPI}, \quad t_j^* t_j = 1_{\TI} \text{ for } j \ge 2, \quad t_1 t_1^* = e, \quad
1_{\TI} + \sum_{j=2}^m t_j t_j^* < 1_{\whatPI} 
\end{gather*}
for $m \ge 2$.  
They go to the relations:
\begin{gather*}
\sum_{i=1}^m T_i T_i^* + t_1 t_1^* <1_{\TI}, \quad 1_{\TI} + \sum_{j=2}^n t_j t_j^* < 1_{\whatPI}
\text{ for } m,n \in \N \text{ with } n\ge 2, \\
t_1^* t_1 = 1_{\whatPI}, \quad
T_i^* T_i = t_j^* t_j = 1_{\TI} \text{ for } i, j \in \N \text{ with } j \ne 1.
 \end{gather*}
Therefore we have the following proposition.
 \begin{proposition}\label{prop:pinftyExelLaca}
The reciprocal dual $\widehat{\calP}_{\infty}$
of $\PI$  is the universal Kirchberg algebra generated by two families of partial isometries 
$\{T_i\}_{i \in \N}$ and $\{ t_j \}_{j \in \N}$ with orthogonal ranges 
which is realized as a simple Exel--Laca algebra $\O_{\widehat{P}_\infty}$
defined by the matrix $\widehat{P}_\infty$ such as   
\begin{equation*}
\widehat{P}_\infty =
\begin{bmatrix}
1 & 1 & 1 &  \cdots & 1 & 0 & 0 & 0 & \cdots \\
1 & 1 & 1 &  \cdots & 1 & 0 & 0 & 0 & \cdots \\
1 & 1 & 1 &  \cdots & 1 & 0 & 0 & 0 & \cdots \\
1 & 1 & 1 &  \cdots & 1 & 0 & 0 & 0 & \cdots \\
\vdots &  \vdots & \vdots & \cdots & \vdots & \vdots & \vdots & \vdots &\cdots \\
1 & 1 & 1 &  \cdots & 1 & 1 & 1 & 1 & \cdots \\
1 & 1 & 1 &  \cdots & 1 & 0 & 0 & 0 & \cdots \\
1 & 1 & 1 &  \cdots & 1 & 0 & 0 & 0 & \cdots \\
1 & 1 & 1 &  \cdots & 1 & 0 & 0 & 0 & \cdots \\
1 & 1 & 1 &  \cdots & 1 & 0 & 0 & 0 & \cdots \\
\vdots &  \vdots & \vdots & \cdots & \vdots & \vdots & \vdots & \vdots &\cdots 
\end{bmatrix}.
\end{equation*}
\end{proposition}


{\bf 3.} The reciprocal dual $\widehat{\O}_\infty$ of $\OI$.

The $C^*$-algebra $\PI$ is the  Exel--Laca algebra 
$\O_{P_\infty}$ for the matrix $P_\infty$ defined by \eqref{eq:Pinftymatrix}.
It is the universal $C^*$-algebra
generated by a family $\{S_i\}_{i\in \N}$ of partial isometries
such that 
\begin{equation*}
\begin{cases}
& S_1^* S_1 + S_2 S_2^* = S_2^* S_2 + S_1 S_1^* =1, \\
& S_3^* S_3 = S_1 S_1^* + S_3 S_3^*, \\
& S_4^* S_4 = S_2 S_2^* + S_4 S_4^*, \\
& S_5^* S_5 = S_3 S_3^* + S_5 S_5^*, \\
& \cdots
\end{cases}
\quad
\text{ and }
\quad
\sum_{j=1}^m S_j S_j^* < 1 
\text{ for } m \in \N.
\end{equation*}
Let us consider the following universal $C^*$-algebra 
written $\T_{P_\infty}$ generated by partial isometries
$s_i, i \in \N$  and a nonzero projection $e$:
\begin{equation}\label{eq:TPI}
\begin{cases}
& s_1^* s_1 + s_2 s_2^* +e = s_2^* s_2 + s_1 s_1^* =1, \\
& s_3^* s_3 = s_1 s_1^* + s_3 s_3^*, \\
& s_4^* s_4 = s_2 s_2^* + s_4 s_4^*, \\
& s_5^* s_5 = s_3 s_3^* + s_5 s_5^*, \\
& \cdots
\end{cases}
\quad
\text{ and }
\quad
e + \sum_{j=1}^m s_j s_j^* < 1 
\text{ for  } m \in \N.
\end{equation}
Since the $C^*$-subalgebra 
$C^*(s_\mu s_2 e s_2^* s_\nu^* \mid \mu, \nu \in \{1,2,3,\dots \}^*)$ 
of $\T_{P_\infty}$ forms an essential ideal of $\T_{P_\infty}$ isomorphic to $\calK(H)$, 
we have a short exact sequence
\begin{equation}\label{eq:OPI}
0
\longrightarrow \calK(H)
\longrightarrow \T_{P_\infty}
\longrightarrow \O_{P_\infty}
\longrightarrow 0.
\end{equation}
As
$ s_1^* s_1 + s_2 s_2^* +e = s_2^* s_2 + s_1 s_1^* =1$
and $[s_is_i^*]_0 = [s_i^* s_i]_0$ in $\K_0(\T_{P_\infty})$, 
we have $[e]_0 = 0$ in  $\K_0(\T_{P_\infty}).$
By the $\K$-theory standard cyclic six term exact sequence for \eqref{eq:OPI}
together with $\K_1(\PI) = \Z, \, \K_0(\PI) = \{ 0 \}$, 
we have the cyclic six-term exact sequence
\begin{equation}\label{eq:standard6termforTPI2} 
\begin{CD}
\K_1(\T_{P_\infty}) @>>> \Z @>{\partial}>> \Z \\
   @AAA @. @VV{\iota_*}V\\
0 @<<< 0 @<<< \K_0(\T_{P_\infty}).
\end{CD}
\end{equation}
Since the downward arrow  
$\iota_*: \Z \rightarrow  \K_0(\T_{P_\infty})$
satisfies $\iota_*(1) = [e]_0 =0 $ in $\K_0(\T_{P_\infty})$,
so that  
$\iota_*: \Z \rightarrow  \K_0(\T_{P_\infty})$
is the zero map.
Hence we have
$\K_0(\T_{P_\infty}) = 0,$
and the upper horizontal arrow
$\partial: \Z \rightarrow \Z$ is surjective
and hence isomorphic.
We then have $\K_1(\T_{P_\infty}) = 0,$
and hence 
\begin{equation*}
(\K_0(\T_{P_\infty}), [e]_0, \K_1(\T_{P_\infty}))
\cong
( 0, \, 0,\,   0).
\end{equation*}
As $\Exts^i(\OI) = 0$ for $i=0,1$, 
we have
\begin{equation*}
(\K_0(\T_{P_\infty}), [e]_0, \K_1(\T_{P_\infty}))
\cong
(\Exts^1(\OI), [\iota_{\OI}]_s, \Exts^0(\OI)).
\end{equation*} 
Hence the $C^*$-algebra $\widehat{\O}_\infty$
is the universal $C^*$-algebra $C^*_{\text{univ}}(\T_{P_\infty}, \{t_i\}_{i \in \N}) $
generated by $\T_{P_\infty}$ and partial isometries $\{t_j \}_{j \in \N}$ satisfying 
\eqref{eq:main} for $\T = \T_{P_\infty}$.
By \cite[Section 8]{MatSogabe3}, 
$\widehat{{\O}}_\infty
$
is isomorphic to the Cuntz algebra $\O_2$ (cf. \cite{Cuntz77}),
one may construct an aperiodic ergodic automorphism on $\O_2$ 
as in Section \ref{sec:ergodicaut}.
We then choose the index set of $\{ t_j\}_{j \in \N}$ as the integer group $\Z$
such that $\{ t_j\}_{j \in \Z}$
satisfies \eqref{eq:univZ}. 
By Theorem \ref{thm:mainuniversalA3} (i) and Proposition \ref{prop:ergodicauto}, 
we may describe our ergodic automorphism on $\O_2$ 
as an automorphism on the universal $C^*$-algebra 
$C^*_{\text{univ}}(\{s_i\}_{i \in \N}, f, \{t_j\}_{j \in \Z})$
in the following way.

\begin{proposition}\label{prop:ergodicO2}
\begin{enumerate}
\renewcommand{\theenumi}{(\roman{enumi})}
\renewcommand{\labelenumi}{\textup{\theenumi}}
\item The Cuntz algebra $\O_2$ is realized as the the universal $C^*$-algebra 
$C^*_{\text{univ}}(\{s_i\}_{i \in \N},  f, \{t_j\}_{j \in \Z})$
generated by two families of partial isometries
$s_i,i \in \N, \,t_j, j \in \Z$  and a nonzero projection $f$
subject to the following operator relations:
\begin{equation}\label{eq:TPI}
\begin{cases}
& t_0^* t_0 =1, \quad t_j^* t_j = f \quad \text{ for } j\ne 0,\\
& f + \sum_{0<|j|<m}t_j t_j^* <1 \quad \text{ for } m \in \N, \\
& t_0 t_0^* + \sum_{i=1}^m s_i s_i^* < f 
\text{ for } m \in \N, \\
& s_1^* s_1 + s_2 s_2^* + t_0 t_0^* =  s_2^* s_2 + s_1 s_1^* =f, \\
& s_3^* s_3 = s_1 s_1^* + s_3 s_3^*, \\
& s_4^* s_4 = s_2 s_2^* + s_4 s_4^*, \\
& s_5^* s_5 = s_3 s_3^* + s_5 s_5^*, \\
& \qquad \vdots
\end{cases}
\end{equation}
\item The correspondence 
$\hat{\theta}:C^*_{\text{univ}}(\{s_i\}_{i \in \N},f, \{t_j\}_{j \in \Z})
\rightarrow C^*_{\text{univ}}(\{s_i\}_{i \in \N}, f, \{t_j\}_{j \in \Z})$
defined by 
\begin{equation*}
\begin{cases}
\hat{\theta}(s_i)  & = t_1 s_i t_1^* \quad \text{ for } i \in \N, \\
\hat{\theta}(f)  & = t_1 f t_1^*, \\
\hat{\theta}(t_0)  & = t_1 t_0, \\
\hat{\theta}(t_{-1})  & = t_1^*, \\
\hat{\theta}(t_j)  & = t_{j+1} t_1^* \quad \text{ for } j \in \Z, j \ne 0, \, -1 
\end{cases}
\end{equation*}
yields an aperiodic ergodic automorphism on $\O_2$.
\end{enumerate}
\end{proposition}

\begin{remark}
 In \cite{CE}, Carey--Evans constructed and studied some ergodic automorphisms on the Cuntz algebras $\O_n$.
 As in \cite[Corollary 3.6]{CE}, their ergodic automorphisms on $\O_2$ have two invariant states,
 whereas our automorphism $\hat{\theta}$ on $\O_2$ has  exactly one invariant state, 
 so that  our automorphism on $\O_2$ is  different from the automorphims in \cite{CE}.
 By \cite{Nakamura}, 
 two aperiodic automorphims are cocycle conjugate 
 if and only if their behaviers on $\K$-theorey groups 
 coincide.
 Since $\K_0(\O_2) = 0$, there exists a unique aperiodic ergodic automorphism on $\O_2$ up to cocycle conjugate.
 Hence our ergodic automorphism on $\O_2$ is cocycle conjugate to the one of Carey--Evans, 
 but not conjugate.      
\end{remark}

\medskip

{\it Acknowledgments:}
The authors would like to thank Masaki Izumi for his useful comments and drawing the authors' attention
to  \cite{CE}.  
K. Matsumoto is supported by JSPS KAKENHI Grant Number 24K06775.
T. Sogabe is supported by JSPS KAKENHI Grant Number 24K16934.


\begin{thebibliography}{99}

\bibitem{AbadieDykema}
{\sc B. Abadie and K. Dykema},
{\it Unique ergodicity of free shifts and some other automorphisms of $C^*$-algebras},
J. Operator Theory {\bf 61}(2009), pp. 279--294.

\bibitem{Blackadar}
{\sc  B. E. Blackadar},
{\it K-theory for operator algebras},
MSRI Publications 5, Springer-Verlag,
 Berlin, Heidelberg and New York, 1986.

\bibitem{BO} 
{\sc N. Brown and N. Ozawa},
{\it  $C^*$-algebras and finite-dimensional approximation}, 
Graduate Studies in Mathematics, 88, American Mathematical Society, 2008.

\bibitem{CE} 
{\sc A. L. Carey and D. E. Evans}, 
{\it On an automorphic action of $U(n, 1)$ on $\O_n$}, 
J. Funct. Anal. {\bf 70}(1987), pp. 90--110.

\bibitem{Cuntz77}{\sc J. Cuntz},
{\it Simple $C^*$-algebras generated by isometries}, 
Commun. Math. Phys. {\bf 57}(1977), pp. 173--185.

\bibitem{Cuntz80}
{\sc J. Cuntz}, 
{\it  A class of $C^*$-algebras and topological Markov chains II: 
reducible chains and the $\Ext$-functor for $C^*$-algebras},
Invent. Math. {\bf 63}(1980), pp.  25--40.

\bibitem{Cuntz1984}
{\sc J. Cuntz},
{\it On the homotopy groups of the space of endomorphisms of a 
$C^*$-algebra (with applications to topological Markov chains)},
In {\it operator algebras and group representations},  Vol I (Neptun 1980), 
pp. 124--137, Pitman, Boston, MA. 1984.

 
\bibitem{CK}{\sc J. ~Cuntz and W. ~Krieger}, 
{\it A class of $C^*$-algebras and topological Markov chains}, Invent.\ Math.\ {\bf 56}(1980), pp.\ 251--268.


\bibitem{Dadarlat2007}{\sc M. Dadarlat},
{\it The homotopy groups of the automorphism groups of Kirchberg algebras},
 J. Noncommut. Geo.
  {\bf 1}(2007), pp.\ 113--139.

\bibitem{EFW}
{\sc M. Enomoto, M. Fujii and Y. Watatani},
{\it Tensor algebras on the Fock space associated with $\OA$}, 
Math. Japon. {\bf 24}(1979), pp. 463--468.

\bibitem{Evans}{\sc D. E. Evans},
{\it Gauge actions on $\OA$},
J. Operator Theory, {\bf 7}(1982), pp. 79--100.


\bibitem{EL}
{\sc R. Exel and M. Laca},
{\it Cuntz--Krieger algebras for infinite matrices},
J. Reine und Angew. Math. {\bf 512}(1999), 119--172.



\bibitem{GS} 
{\sc J. Gabe and G. Szab\'{o}}, 
{\it The dynamical Kirchberg--Phillips theorem}, Acta Math. {\bf 232}(2024), pp. 1--77.


\bibitem{Izumi}
{\sc M. Izumi},
{\it  Finite group actions on $C^*$-algebras with the Rohlin property. I},
Duke Math. J. {\bf 122}(2004), pp.  233-–280.



\bibitem{KP}
{\sc J. Kaminker and I. F. Putnam},
{\it K-theoretic duality for shifts of finite type},
Commun. Math. Phys.
{\bf 187}(1997), pp. 509--522.

\bibitem{KS}
{\sc J. Kaminker and C. Schochet},
{\it Spanier-Whitehead $\K$-duality for $C^*$-algebras},
J. Topol. Anal. {\bf 11}(2019), pp.  21--52.


\bibitem{Kirchberg}
{\sc E. Kirchberg},
{\it The classification of purely infinite $C^*$-algebras using Kasparov's theory}, preprint, 1994.

\bibitem{Kumjian}
{\sc A. Kumujian},
{\it On certain Cuntz--Pimsner algebras},
Pacific J. Math. {\bf 217}(2004), pp. 275--289.


\bibitem{MaJMAA2024}{\sc K. Matsumoto},
{K-theoretic duality for extensions of  Cuntz--Krieger algebras},
J. Math. Anal. Appl.
{\bf 531}(2024), P. 127827, 26 pages.


\bibitem{MaAnalMath2024}
{\sc K. Matsumoto},
{\it On strong extension groups of Cuntz--Krieger algebras},
Anal. Math. {\bf 50}(2024), pp. 917--937.



\bibitem{MatSogabe}{\sc K. Matsumoto and T. Sogabe},
{\it On the homotopy groups of the automorphism groups of Cuntz--Krieger algebras},
 preprint,   arXiv:2404.06115 [math.OA], to appear in J.  Noncommut. Geo., published online first
DOI 10.4171/JNCG/598.



\bibitem{MatSogabe2}{\sc K. Matsumoto and T. Sogabe},
{\it Total extension groups for unital Kirchberg algebras},
 Math. Scand. {\bf 131}(2025), pp.  367--390.


\bibitem{MatSogabe3}
{\sc K. Matsumoto and T. Sogabe}, 
{\it Reciprocal Cuntz--Krieger algebras}, 
Münster J. Math. {\bf 18}(2025), pp. 249--291.


\bibitem{MatSogabe4}
{\sc K. Matsumoto and T. Sogabe}, 
{\it Ergodic automorphisms on unital Kirchberg algebras},
 preprint,  arXiv:2505.23168 [math.OA].


\bibitem{Nakamura}
{\sc H. Nakamura},
{\it Aperiodic automorphisms on simple purely infinite $C^*$-algebras},
Ergodic Theory Dynam. Systems {\bf 20}(2000), 1749--1765.

 
\bibitem{PennigSogabe}
{\sc U. Pennig and T. Sogabe},
{\it Spanier--Whitehead $\K$-duality 
and duality of extensions of C*-algebras}, 
Internat. Math. Res. Notices. {\bf 2024}(2024),
pp. 14321--14351. 

\bibitem{Phillips}
{\sc N. C. Phillips},
{\it A classification theorem for nuclear purely infinite simple $C^*$-algebras}, 
Doc. Math. {\bf 5}(2000), pp. 49--114.

\bibitem{Pimsner} 
{\sc M. V. Pimsner},
{\it  A class of $C^*$-algebras generalizing both Cuntz--Krieger algebras and crossed products by $\mathbb{Z}$},
 Fields Institute Communications, {\bf 12}(1997), pp. 189--212.


\bibitem{RaebSzy} {\sc I. Raeburn and W. Szyma{\'n}sky},
{\it Cuntz--Krieger algebras of infinite graphs and matrices},
 Trans. Amer. Math. Soc. {\bf 356}(2004), pp. 39--59.



\bibitem{Skandalis}
{\sc G. Skandalis},
{\it On the strong Ext bifunctor}, 
Preprint, Queen's University no 19, 1983.


\bibitem{Sogabe2022}{\sc T. Sogabe},
{\it The reciprocal Kirchberg algebras},
J. Func. Anal. {\bf 286}(2024), Paper No. 110395, 53 pp.

\bibitem{Takesaki}
{\sc M. Takesaki},
{\it Theory of Operator algebras I},
Springer-Verlag, Berlin, Heidelberg and New York, 1979.


\end{thebibliography}
\end{document}